\documentclass[reqno,10pt,twoside]{amsart}
\numberwithin{equation}{section}

\usepackage{fullpage}
\usepackage{xypic} 
\usepackage{graphicx}   
\usepackage{float}
\usepackage{graphics,amssymb}  
\usepackage{enumerate}
\usepackage{latexsym}
\usepackage{epsf}
\usepackage{mathrsfs}
\usepackage{bbm}
\usepackage{hyperref} 
\hypersetup{ 
colorlinks,
citecolor=black,
filecolor=blue,
linkcolor=black,
urlcolor=blue 
} 
\usepackage{comment}
\xyoption{curve}
\newtheorem{lemma}{Lemma}[section]
\newtheorem{lem/def}[lemma]{Lemma/Definition}
\newtheorem{proposition}[lemma]{Proposition}
\newtheorem{theorem}[lemma]{Theorem}
\newtheorem{corollary}[lemma]{Corollary}

\theoremstyle{definition}

\newtheorem{definition}[lemma]{Definition}
\newtheorem{remark}[lemma]{Remark}

\newtheorem{notation}[lemma]{Notation}

\newcommand{\beq}{\begin{equation*}}
\newcommand{\eeq}{\end{equation*}}
\newcommand{\beqlbl}{\begin{equation}}
\newcommand{\eeqlbl}{\end{equation}}
\newcommand{\ba}{\begin{array}}
\newcommand{\ea}{\end{array}}

\newcommand{\Hom}{\mathrm{Hom}}
\newcommand{\RHom}{\mathrm{RHom}}
\newcommand{\Ext}{\mathrm{Ext}}

\newcommand{\End}{\mathrm{End}}
\newcommand{\gr}{\mathrm{gr}}
\newcommand{\ox}{\otimes}

\newcommand{\xym}{\xymatrix}

\renewcommand{\k}{\mathbf{k}}

\title[Cohomology Rings of Smash Products]{Spectral Sequences for the
  Cohomology Rings of a Smash Product}
\date{}

\author{Cris Negron}
\email{negron@uw.edu}
\address{Department of Mathematics, University of Washington, Seattle, WA 98195, USA}
\thanks{This material is based upon work supported by the National Science Foundation Graduate Research Fellowship under Grant No. DGE-1256082}

\begin{document}
\maketitle

\begin{abstract}
Stefan and Guichardet have provided Lyndon-Hochschild-Serre type
spectral sequences which converge to the Hochschild cohomology and Ext
groups of a smash product.  We show that these spectral sequences
carry natural multiplicative structures, and that these multiplicative
structures can be used to calculate the cup product on Hochschild
cohomology and the Yoneda product on an Ext algebra.
\end{abstract}

\section{Introduction}

Let $\k$ be a field of arbitrary characteristic, and take
$\ox=\ox_\k$.  Fix a ($\k$-)algebra $A$ and a Hopf algebra $\Gamma$
acting on $A$.  We assume that the antipode on $\Gamma$ is bijective,
and that the action of $\Gamma$ on $A$ gives it the structure of a
$\Gamma$-module algebra \cite[Definition 4.1.1]{M}.  We can then form
the smash product algebra $A\#\Gamma$, which is the vector space $A\ox
\Gamma$ with multiplication
\beq
(a\ox\gamma)(a'\ox\gamma'):=\sum_ia(\gamma_{i_1}\cdot a')\ox \gamma_{i_2}\gamma',
\eeq
where $\Delta(\gamma)=\sum_i\gamma_{i_1}\ox\gamma_{i_2}$.  In the case
that $\Gamma$ is the group algebra of a group $G$ we use the notation
$A\# G$ as a shorthand for the smash product $A\#\k G$.
\par
Smash products have appeared in a number of different contexts in the
literature.  Under certain conditions, the smash product can serve as
a replacement for the invariant algebra $A^\Gamma$ \cite[Section
4.5]{M} \cite{AFLS} \cite[Proposition 5.2]{CI}.  Smash products have
also appeared as a means of untwisting, or unbraiding, certain twisted
structures.  For example, one can untwists a twisted Calabi-Yau
algebra with a smash product \cite{RRZ, GK}, or unbraid a braided Hopf
algebra \cite[Section 1.5]{AS}.  In a more classical context, smash
products have played an integral role in a classification program
proposed by Andruskiewitsch and Schneider, which began at \cite{AS98}.
\par
In this paper we equip some known spectral sequences, which converge
to the Hochschild cohomology and Ext groups of a smash product, with
multiplicative structures.  These spectral sequences, with their new
multiplicative structures, can then be used to compute the products on
these cohomologies.  Specifically, we provide spectral sequences which
converge to the Hochschild cohomology $\mathrm{HH}(A\#\Gamma, B)$,
along with the cup product, and the extension algebra
$\Ext_{A\#\Gamma\text{-mod}}(M,M)$, along with the standard Yoneda
product.  Here we allow $M$ to be any (left) $A\#\Gamma$-module and
$B$ to be any algebra extension of the smash product, i.e. any algebra
equipped with an algebra map $A\#\Gamma\to B$.
\par
Recall that the Hochschild cohomology $\mathrm{HH}(R,M)=\oplus_i\
\mathrm{HH}^i(R,M)$ of an algebra $R$, with ``coefficients'' in a
$R$-bimodule $M$, is defined to be the graded Ext group
$\mathrm{HH}(R,M)=\Ext_{R\text{-bimod}}(R,M)$.  In the case that $M$
is an algebra extension of $R$, the Hochschild cohomology
$\mathrm{HH}(R,M)$ carries a natural product called the cup product.
The definition of the cup product is reviewed in Section \ref{rmind}.
\par
The Hochschild cohomology ring $\mathrm{HH}(R,R)$ is known to be an
invariant of the derived category $D^b(R)$ \cite{rickard91}.  In
addition to providing a relatively refined derived invariant, the cup
product can also help us to analyze the module category of a given
algebra.  Snashall and Solberg have put forward a theory of support
varieties for Artin algebras by way of the Hochschild cohomology ring
$\mathrm{HH}(R,R)$.  They assign to an Artin algebra $R$, and any pair
of $R$-modules, a subvariety of the maximal ideal spectrum of (a
subalgebra of) the Hochschild cohomology.  The cup product can also
help us get a handle on some of the additional structures on
Hochschild cohomology, such as the Gerstenhaber bracket.  The
applications of Hochschild cohomology rings are, however, limited by a
scarcity of examples and by the fact that the cup product can be
difficult to compute in general.
\par
In the theorem below, by a multiplicative spectral sequence we mean a
spectral sequence $E=(E_r)$ equipped with bigraded products $E_r\ox
E_r\to E_r$ which are compatible with the differentials and structural
isomorphism $E_{r+1}\cong \mathrm{H}(E_r)$.  (One can refer to Section
\ref{DIA} for a more precise definition.)  We say that a
multiplicative spectral sequence converges to a graded algebra
$\mathrm{H}$ if $\mathrm{H}$ carries an additional filtration and
there is an isomorphism of bigraded algebras $E_\infty\cong \gr
\mathrm{H}$.  One of our main result is the following.

\begin{theorem}[Corollary \ref{spectrill}]
For any algebra extension $B$ of the smash product $A\#\Gamma$, there
are two multiplicative spectral sequences
\beq
E_2=\Ext_{\Gamma\text{-}\mathrm{mod}}(\k,\mathrm{HH}(A,B))\Rightarrow \mathrm{HH}(A\#\Gamma,B)
\eeq
and
\beq
'E_1=\Ext_{\Gamma\text{-}\mathrm{mod}}(\k,\RHom_{A\text{-}\mathrm{bimod}}(A,B))\Rightarrow \mathrm{HH}(A\#\Gamma,B)
\eeq
which converge to the Hochschild cohomology as an algebra.
\label{thm0}
\end{theorem}

To be clear, we mean that there is some $\Gamma$-module algebra
structure on $\mathrm{HH}(A,B)$ and that the second term $E_2$ is the
bigraded algebra
$\Ext_{\Gamma\text{-mod}}(\k,\mathrm{HH}(A,B))$. Similarly, for the
term $'E_1$, we mean there is some particular model for
$\RHom_{A\text{-bimod}}(A,B))$ which is a $\Gamma$-module (dg) algebra
and that $'E_1$ is the given Ext algebra.  We also provide a version
of the above theorem for the cohomology rings $\Ext_{A\#\Gamma\text{-mod}}(M,M)$.

\begin{theorem}[Corollary \ref{lastocor}]
For any $A\#\Gamma$-module $M$, there are two multiplicative spectral sequences
\beq
\bar{E}_2=\Ext_{\Gamma\text{-}\mathrm{mod}}(\k,\Ext_{A\text{-}\mathrm{mod}}(M,M))\Rightarrow \Ext_{A\#\Gamma\text{-}\mathrm{mod}}(M,M)
\eeq
and
\beq
'\bar{E}_1=\Ext_{\Gamma\text{-}\mathrm{mod}}(\k,\RHom_{A\text{-}\mathrm{mod}}(M,M))\Rightarrow \Ext_{A\#\Gamma\text{-}\mathrm{mod}}(M,M)
\eeq
which converge to $\Ext_{A\#\Gamma\text{-}\mathrm{mod}}(M,M)$ as an algebra.
\label{thm00}
\end{theorem}

In the text it is shown that all four of the above spectral sequences
exist as explicit isomorphism at the level of cochains.  Let us
explain what is meant by this statement in the case of Hochschild cohomology.  
\par
Let $B$ be an algebra extension of $A\#\Gamma$, as in Theorem
\ref{thm0}.  For a free $A$-bimodule resolution $K\to A$, equipped
with a $\Gamma$-action satisfying certain natural conditions, and any
resolution $L\to \k$ of the trivial $\Gamma$-module, we produce a dg
algebra structure on the double complex
\beqlbl
\Hom_{\Gamma\text{-mod}}(L,\Hom_{A\text{-bimod}}(K,B)).
\label{dcmplx}
\eeqlbl
From this data we also produce a $A\#\Gamma$-bimodule resolution
$\mathcal{K}$ of $A\#\Gamma$, and dg algebra structure on the
associated complex $\Hom_{A\#\Gamma\text{-bimod}}(\mathcal{K},B)$.
The dg algebra structure is chosen so that the homology of
$\Hom_{A\#\Gamma\text{-bimod}}(\mathcal{K},B)$ is the Hochschild
cohomology $\mathrm{HH}(A\#\Gamma, B)$ with the cup product.
\par
In Theorem \ref{bigthm2}, which can be seen as a lifting of Theorem
\ref{thm0} to the level of cochains, we show that there is an explicit
isomorphism of dg algebras
\beq
\Hom_{A\#\Gamma\text{-bimod}}(\mathcal{K},B)\overset{\cong}\to
\Hom_{\Gamma\text{-mod}}(L,\Hom_{A\text{-bimod}}(K,B)).
\eeq
It follows then that the Hochschild cohomology ring of the smash
product can be computed as the homology of the double complex
$\Hom_{\Gamma\text{-mod}}(L,\Hom_{A\text{-bimod}}(K,B))$.  We get
Theorem \ref{thm0} as an easy corollary of this fact.
\par
The full power of Theorem \ref{bigthm2} is employed to compute some
examples in a follow up paper.  Let us discuss just one example here.
Suppose $\k$ is of characteristic $0$ and let $q\in\k$ be a nonzero
scalar which is not a root of unity.  Let $\k_q[x,y]$ denote the skew
polynomial ring in $2$-variables,
\beq
\k_q[x,y]=\frac{\k\langle x,y\rangle}{(yx-qxy)}.
\eeq
\par
This algebra is twisted Calabi-Yau.  If we let
$\mathbb{Z}=\langle\phi\rangle$ act on $\k_q[x,y]$ by the automorphism
$\phi:x\mapsto q^{-1}x$, $y\mapsto qy$, then, according to
\cite[Proposition 7.3]{RRZ} and \cite{GK}, the smash product
$\k_q[x,y]\#\mathbb{Z}$ will be  Calabi-Yau.  By way of Theorem
\ref{bigthm2}, we can provide the following computation. 

\begin{theorem}
Let $\lambda$, $\varepsilon$, $\xi_i$, $\zeta$, and $\eta_i$ be a
variables of respective degrees $0$, $1$, $1$, $2$, and $2$.  Then
there is an isomorphism of graded algebras
\beq
\mathrm{HH}(\k_q[x,y]\#\mathbb{Z})\cong\frac{\k[\varepsilon,\lambda,\xi_1,\xi_2,\zeta]}{(\lambda\zeta-\xi_1\xi_2,\zeta^2)}\times_{\k[\varepsilon]}\frac{\k[\varepsilon,\eta_i:i\in\mathbb{Z}-\{-1\}]}{(\eta_i^2,\eta_i\eta_j)}.
\eeq
Furthermore, there is a natural embedding of graded algebras $\mathrm{HH}(\k_q[x,y])\to \mathrm{HH}(\k_q[x,y]\#\mathbb{Z})$ identifying $\mathrm{HH}(\k_q[x,y])$ with the subalgebra generated by $\xi_1$, $\xi_2$, and $\eta_0$.
\label{lastthm}
\end{theorem}
In the statement of the above theorem $\k[X_1,\dots, X_n]$ denotes the
free graded commutative algebra on graded generators $X_i$, $\k[X_1,\dots, X_n]=\k\langle X_1,\dots,X_n\rangle/(X_iX_j-(-1)^{|X_i||X_j|}X_jX_i)$.

\subsection{Relation to the work of Stefan, Guichardet, and others}
As suggested in the abstract, Theorems \ref{thm0} and \ref{thm00} can
be seen as a refinement of results of Stefan and Guichardet given in
\cite{stefan} and \cite{guichardet01} respectively.  However, both
Stefan and Guichardet work with classes of
algebras that are slightly different than general smash products.  Guichardet provides spectral sequences 
\beq
\Ext_{\k G\text{-}\mathrm{mod}}(\k,\mathrm{HH}(A,M))\Rightarrow
\mathrm{HH}(A\#_{\alpha}G, M)
\eeq
for crossed product algebras $A\#_{\alpha}G$, where $G$ is a group, while
Stefan provides spectral sequences
\beq
\Ext_{\Gamma\text{-}\mathrm{mod}}(\k,\mathrm{HH}(A,M))\Rightarrow \mathrm{HH}(E, M)
\eeq
 for Hopf Galois extensions $A\to E$.  Guccione and Guccione extend
 the results of Guichardet to allow for crossed products with
 arbitrary Hopf algebras in \cite{GuGu}.  None of these spectral sequences carry any
 multiplicative structures.  For definitions of these different
 classes of algebras one can see \cite{M}.  Let us only mention that
 there are strict containments
\beq
\{\text{Smash Products}\}\subsetneq \{\text{Crossed Products}\}\subsetneq \{\text{Hopf Galois Extensions}\}.
\eeq
So these more limited results (taken together) do apply to larger classes of algebras.  
\par
Let us mention here that Guccione and Guccione also provide
spectral sequences for Ext groups in \cite{GuGu}, by way of a standard relation
\cite[Lemma 9.1.9]{weibel}.  Also, some results involving multiplicative
structures are given by Sanada, in a rather constrained setting, in
\cite{sanada93}.  Further analysis of the situation can be found
in \cite{barnes03}.
\par
The main point of comparison here is that our spectral sequences can
be used to compute the cup product, while those of Stefan, Guichardet, and
Guccione-Guccione can not (at least after restricting to the case of
smash products).  However, there are also differences in the methods
used in the three sources.  As a consequence, the usefulness of the
results vary in practice.  For example, Stefan produces his spectral
sequence as a Grothendieck spectral sequence, whereas those of
Guccione and Guccione are derived from filtrations on a certain (rather
large) complex.  Guichardet shows that the Hochschild cohomology of
a crossed product can be computed by the
double complex $C(G,C(A,M))$, where $C$
denotes the standard Hochschild cochain complex.  Indeed, Guichardet
provides quasi-isomorphisms of chain complexes $C(G,C(A,M))\leftrightarrows C(A\#_\alpha
G,M)$.  It does not appear that either of the given maps are dg
algebra maps, and so the cup product remains obscured.
\par
The spectral sequences produced in this
paper are those associated to the first quadrant double complex
(\ref{dcmplx}), which
may in some cases be chosen to be relatively small.  Our methods are
most closely related those of Guichardet.  In fact, by standard
techniques, one may move from Guichardet's double Hochschild cochain complex to our
double complex(es).  To summarize the
situation, we have the following chart
\\\\
{\Small
\begin{tabular}{|l|c|c|c|}
\hline
& \begin{tabular}{c}Class of algebras for which\\the spectral
  sequences apply\end{tabular}& Type of spectral sequences
& \begin{tabular}{c} Accounts for the\\ cup/Yoneda product\end{tabular}\\\hline
Stefan& Hopf Galois extensions & Grothendieck & No\\\hline
Guichardet & \begin{tabular}{c} Crossed products\\with
  groups\end{tabular} & Double complex & No\\\hline
Guccione-Guccione& Crossed products & Filtration & No\\\hline
Present paper& Smash products & Double complex & Yes\\
\hline
\end{tabular}}
\par
Of the works discussed, Theorems \ref{bigthm1}, \ref{bigthm2}, and
\ref{bigthm3} below provide the most computationally accessible
approach to the cohomology of a smash product, irrespective of the cup
product.  Grothendieck spectral sequences, for example, require the
use of injective resolutions, which are very difficult to come by in
general.  The methods used here are also more natural than those given
in \cite{GuGu} in the sense that many of the constructions we employ are functorial.

\subsection{Contents}

Throughout we consider a Hopf algebra $\Gamma$ acting on an algebra
$A$.  In Section \ref{bmd} we produce a resolution of the Hopf algebra
$\Gamma$ which carries enough structure to admit a smash product construction.  In
particular, we construct a complex of projective $\Gamma$-bimodules with an
additional (compatible) coaction, and quasi-isomorphism to $\Gamma$
which preserves the given structure.  We call such a resolution a Hopf
bimodule resolution.
\par
In Section \ref{bimodrez} we propose a
smash product construction for complexes of Hopf bimodules and
complexes of, so called, equivariant bimodules over $A$ (Definition
\ref{equiv}).  This smash product construction for complexes is used
to produce, from the Hopf
bimodule resolution of Section \ref{bmd} and an equivariant resolution
of $A$, a bimodule resolution of $A\#\Gamma$.  In Section \ref{HCvDI}
we use the aforementioned bimodule resolution of
$A\#\Gamma$ to construct an explicit isomorphism 
\beq
\Xi:\RHom_{A\#\Gamma\text{-bimod}}(A\#\Gamma,M)\overset{\cong}\to \RHom_{\Gamma\text{-mod}}(\k,\RHom_{A\text{-bimod}}(A,M))
\eeq
for any complex $M$ of $A\#\Gamma$-bimodules.
\par
In Section \ref{rmind} we review the products on both the domain and
codomain of the above isomorphism $\Xi$ (when evaluated at an algebra
extension of $A\#\Gamma$), and in Section \ref{DIA} we show that the map $\Xi$
is an isomorphism of dg algebras when appropriate.  Theorem \ref{thm0}
is also proved in this section.  Finally, in Section \ref{coHOM!} we give versions of our main theorems
for the Ext algebras $\Ext_{A\#\Gamma\text{-mod}}(M,M)$, for arbitrary
$M$.

\subsection*{Acknowledgments}

I would like to thank Jiafeng Lu, Xuefeng Mao, Xingting Wang, and
James Zhang for offering many valuable suggestions and corrections
during the writing of this paper.

\section*{Conventions}
For any coalgebra $\Gamma$, the coproduct of an element $\gamma\in
\Gamma$ will always be expressed using Sweedler's notation 
\beq
\gamma_1\ox \gamma_2=\Delta(\gamma).
\eeq
(So ``$\gamma_1\ox \gamma_2$'' is a symbol representing a sum of
elements $\sum_i\gamma_{i_1}\ox \gamma_{i_2}$ in $\Gamma\ox \Gamma$.)
Given $\Gamma$-modules $M$ and $N$ the tensor product $M\ox N$ is
taken to be a $\Gamma$-module under the standard action
\beq
\gamma(m\ox n):=(\gamma_1m)\ox (\gamma_2n).
\eeq
\par
As mentioned previously, {\it a Hopf algebra will always mean a Hopf
  algebra with bijective antipode}.  Let $\Gamma$ be a Hopf algebra
and $A$ be a $\Gamma$-module algebra.  Following \cite{SW}, we denote
the action of $\Gamma$ on $A$ by a superscript ${^\gamma
  a}:=\gamma\cdot a$.  Elements in the smash product $A\#\Gamma$ will
be denoted by juxtaposition $a\gamma:=a\ox \gamma\in A\#\Gamma=A\ox
\Gamma$.  Hence, the multiplication on $A\#\Gamma$ can be written
$(a\gamma)(b\gamma')=a({^{\gamma_1} b})\gamma_2\gamma'$.  For an
algebra $A$ we let $A^e$ denote the enveloping algebra $A^e=A^{op}\ox
A$.  {\it All modules are left modules unless stated otherwise}.  We
do not distinguish between the category of $A$-bimodules and the
category of (right or left) $A^e$-modules.
\par
In computations, all elements in graded vectors spaces are chosen to
be homogenous.  For homogeneous $x$, in a graded space $X$, we let
$|x|$ denote its degree.  For any algebra $A$ and $A$-complexes $X$
and $Y$ we let $\Hom_A(X,Y)$ denote the standard hom complex.  Recall
that the $n$th homogenous piece of the hom complex consists of all
degree $n$ maps $f:X\to Y$, and for any $f\in \Hom_A(X,Y)$ the
differential is given by $f\mapsto d_Yf-(-1)^{|f|}f d_X$.

\section{A Hopf bimodule resolution of $\Gamma$}
\label{bmd}

Let $\Gamma$ be a Hopf algebra.  We have the canonical algebra embedding
\beq
\ba{c}
\Delta^{tw}:\Gamma\to\Gamma^{e}=\Gamma^{op}\ox \Gamma\\
\gamma\mapsto S(\gamma_1)\ox\gamma_2.
\ea
\eeq
This map will be referred to as the {\it twisted diagonal map}.  The
twisted diagonal map gives $\Gamma^{e}$ a left $\Gamma$-module
structure.  On elements, this left action is given by
$\delta\cdot(\gamma\ox\gamma')=\gamma S(\delta_1)\ox\delta_2\gamma'$,
for $\delta\in \Gamma$, $\gamma\ox\gamma'\in\Gamma^e$.
\par
Note that, since the antipode of $\Gamma$ is bijective, there is an
isomorphism of left $\Gamma$-modules $S\ox id:\Gamma^e\to \Gamma\ox
\Gamma$.  The module $\Gamma\ox\Gamma$ is known to be free over
$\Gamma$.  (One can use the fundamental theorem of Hopf modules
\cite[Theorem 1.9.4]{M} to show this, for example.)  So we get the following

\begin{lemma}
The enveloping algebra $\Gamma^{e}$ is a free, and hence flat, left $\Gamma$-module.
\label{freelem}
\end{lemma}

Note that the left action of $\Gamma$ on $\Gamma^e$ is compatible with
the standard (outer) bimodule structure on
$\Gamma^e={_\Gamma\Gamma}\ox\Gamma_{\Gamma}$.  Indeed, the module
structure induced by the twisted diagonal map utilizes the inner
bimodule structure $\Gamma_{\Gamma}\ox{_\Gamma\Gamma}$ exclusively.
So we see that $\Gamma^e$ is a ($\Gamma$-$\Gamma^{e}$)-bimodule, and
that the induced module $M\ox_\Gamma\Gamma^e$ of a right
$\Gamma$-module $M$ is a $\Gamma$-bimodule.  To be clear, the left and
right actions of $\Gamma$ on $M\ox_\Gamma\Gamma^e$ are given by
\beq
\delta\cdot(m\ox_\Gamma(\gamma\ox\gamma')):=m\ox_\Gamma(\delta\gamma\ox\gamma')
\eeq
and
\beq
(m\ox_\Gamma(\gamma\ox\gamma'))\cdot\delta:=m\ox_\Gamma(\gamma\ox\gamma'\delta)
\eeq
respectively, where $\delta\in\Gamma$ and
$m\ox_\Gamma(\gamma\ox\gamma')\in M\ox_\Gamma\Gamma^e$.  The same
analysis holds when we replace $M$ with a complex of right $\Gamma$-modules.

\begin{notation}
Given any right $\Gamma$-module (resp. complex) $M$, we let we let
$M^\uparrow$ denote the induced module (resp. complex) $M\ox_\Gamma\Gamma^{e}$.
\end{notation}

The following result is proven, in less detail, in \cite[Section
3]{SiW}.  However, as we will be needing all the details, a full proof
is given here.

\begin{lemma}
Let $\xi:L\to \k$ be a resolution of the trivial right $\Gamma$-module
$\k=\Gamma/\ker\epsilon$.
\begin{enumerate}
\item The induced complex $L^\uparrow$ is a complex of projective $\Gamma$-bimodules.
\item The map $\xi^{\uparrow}:L^\uparrow\to \Gamma$,
  $\ell\ox_\Gamma(\gamma\ox\gamma')\mapsto \xi(\ell)\gamma\gamma'$, is
  a quasi-isomorphism of complexes of $\Gamma$-bimodules.
\end{enumerate}
\label{setup}
\end{lemma}
Statements (1) and (2) together say that $L^\uparrow$ is a projective
bimodule resolution of  $\Gamma$.
\begin{proof}
In each degree $i$ we have the adjunction
\beq
\Hom_{\Gamma^{e}}( L^i\ox_\Gamma\Gamma^{e},-)=\Hom_\Gamma(L^i,\Hom_{\Gamma^{e}}(\Gamma^{e},-)).
\eeq
Whence the functor on the left is seen to be exact.  So $L^{\uparrow}$ is a complex of projective bimodules.
\par
As an intermediate step in proving (2), let us consider the $\Gamma$-bimodule map
\beq
\ba{c}
\varphi:\k\ox_\Gamma\Gamma^{e}\to \Gamma\\
1\ox_\Gamma(\gamma\ox\gamma')\mapsto \gamma\gamma'.
\ea
\eeq
 For any $\gamma\in\Gamma$ we have $\varphi(1\ox_\Gamma(1\ox\gamma))=\gamma$.  So $\varphi$ is surjective.  Also, the computation
\beqlbl
\ba{rl}
1\ox_\Gamma(\gamma\ox\gamma') & = 1\ox(\gamma_1\epsilon(\gamma_2)\ox \gamma')\\
& =\epsilon(\gamma_2)\ox_\Gamma(\gamma_1\ox\gamma')\\
&=1\cdot\gamma_2\ox_\Gamma(\gamma_1\ox\gamma')\\
&=1\ox_\Gamma(\gamma_1S(\gamma_2)\ox\gamma_3\gamma')\\
&=1\ox_\Gamma(\epsilon(\gamma_1)\ox\gamma_2\gamma')\\
&=1\ox_\Gamma(1\ox\gamma\gamma')
\ea
\label{cmptn}
\eeqlbl
makes it clear that any element in $\k\ox_\Gamma\Gamma^e$ is of the
form $1\ox_\Gamma(1\ox\gamma)$ for some $\gamma$.  Whence $\varphi$ is
seen to be injective as well, and therefore an isomorphism of
bimodules.  Now, for (2), simply note that
$\xi\ox_\Gamma\Gamma^{e}:L^\uparrow\to\k\ox_\Gamma(\Gamma^{e})$ is a
quasi-isomorphism, since $\xi$ is a quasi-isomorphism and $\Gamma^{e}$
is flat over $\Gamma$, and that $\xi^\uparrow$ can be given as the
composition of the isomorphism $\varphi$ with the quasi-isomorphism $\xi\ox_\Gamma\Gamma^{e}$.
\end{proof}
 
\begin{definition}
Given any right $\Gamma$-module $M$, we define the left
$\Gamma$-comodule structure $\rho_M$ on $M^{\uparrow}$ by
\beqlbl
\ba{c}
\rho_M:M^\uparrow\to \Gamma\ox M^{\uparrow}\\
m\ox_\Gamma(\gamma\ox\gamma')\mapsto (\gamma_1\gamma_1')\ox(m\ox_\Gamma(\gamma_2\ox\gamma'_2)).
\ea
\label{blah}
\eeqlbl
Following the standard notation, for any $\mathrm{m}\in M^\uparrow$,
we denote the element $\rho_M(\mathrm{m})$ by $\mathrm{m}_{-1}\ox \mathrm{m}_0$.
\label{codef}
\end{definition}
There is something of a question of whether or not this coaction is
well defined.  Certainly we can define a coaction on the $\k$-tensor product
\beq
\tilde{\rho}_M:M\ox \Gamma^{e}\to \Gamma\ox(M\ox_\Gamma\Gamma^{e})
\eeq
by the same formula as (\ref{blah}).  A direct computation shows that
$\tilde{\rho}_M\big(m\delta\ox(\gamma\ox\gamma')-m\ox\delta(\gamma\ox\gamma')\big)=0$,
i.e. that $\tilde{\rho}_M$ vanishes on the relations for the tensor
product $M\ox_\Gamma \Gamma^e=M^{\uparrow}$.  Whence the coaction
$\rho_M$ can be given as the map induced on the quotient $M^\uparrow=M\ox_\Gamma \Gamma^e$.  

\begin{definition}
By a {\it Hopf bimodule} we will mean a $\Gamma$-bimodule $N$ equipped
with a left $\Gamma$-coaction $N\to \Gamma\ox M$ which is a map of
$\Gamma$-bimodules (where $\Gamma$ acts diagonally on the tensor
product $\Gamma\ox N$).  Maps of Hopf bimodules are maps that are
simultaneously $\Gamma$-bimodule maps and $\Gamma$-comodule maps.
\end{definition}

The algebra $\Gamma$ itself becomes a Hopf bimodule under the regular
bimodule structure and coaction given by the comultiplication.  In the
notation of \cite[Section 1.9]{M}, a Hopf bimodule is an object in the
category $_\Gamma^{\Gamma}\mathcal{M}_\Gamma$.

\begin{proposition}
Let $M$ and $N$ be right $\Gamma$-modules and $f:M\to N$ be a morphism
of $\Gamma$-modules .  For any $\mathrm{m}\in M^{\uparrow}$ and
$\gamma\in \Gamma$ the following equations hold:
\begin{enumerate}
\item $\rho_N(f^\uparrow(\mathrm{m}))=\mathrm{m}_{-1}\ox f^\uparrow(\mathrm{m}_0)$
\item $\rho_M(\mathrm{m}\cdot\gamma)=\mathrm{m}_{-1}\gamma_1\ox\mathrm{m}_0\gamma_2$
\item $\rho_M(\gamma\cdot \mathrm{m})=\gamma_1\mathrm{m}_{-1}\ox\gamma_2\mathrm{m}_0$.
\end{enumerate}
Said another way, $(-)^\uparrow$ is a functor from
$\mathrm{mod}$-$\Gamma$ to the category of Hopf bimodules.
\end{proposition}

The reader should be aware that we will be using the $^\uparrow$
notation on maps in a slightly more flexible manner throughout the
paper.  We will be generally be looking at the induced map composed
with some convenient isomorphism.

\begin{proof}
These can all be checked directly from the definitions.  For example, for (1), we have
\beq
\ba{rl}
\rho_M(f^\uparrow(m\ox_\Gamma(\gamma\ox\gamma')))&=\rho_M(f(m)\ox_\Gamma(\gamma\ox\gamma'))\\
&=\gamma_1\gamma'_1\ox f(m)\ox_\Gamma(\gamma_2\ox\gamma_2')=\gamma_1\gamma'_1\ox f^\uparrow(m\ox_\Gamma(\gamma_2\ox\gamma_2')).
\ea
\eeq
\end{proof}

\begin{corollary}
For any complex $X$ of right $\Gamma$-modules the induced complex $X^\uparrow$ is a complex of Hopf bimodules.
\end{corollary}
\begin{proof}
This follows from part (1) of the pervious proposition and the fact that the differentials on $X$ are $\Gamma$-linear.
\end{proof}

\begin{proposition}
The quasi-isomorphism $\xi^{\uparrow}:L^{\uparrow}\to \Gamma$ of Lemma \ref{setup} is a quasi-isomorphism of complexes of Hopf bimodules.
\label{propH}
\end{proposition}
\begin{proof}
This can be checked directly from the definition of $\xi^{\uparrow}$ and the definitions of the coactions on $L^\uparrow$ and $\Gamma$.
\end{proof}

\section{Bimodule resolutions of $A\#\Gamma$ via a smash product construction}
\label{bimodrez}

Let $\Gamma$ be a Hopf algebra and $A$ be a $\Gamma$-module algebra.
We recall here that a $\k$-linear map $M\to N$ of (right or left)
$A\#\Gamma$-modules is $A\#\Gamma$-linear if and only if it is
$A$-linear and $\Gamma$-linear independently.  The following
definition was given by Kaygun in \cite{kaygun}.

\begin{definition}
A vector space $M$ is called a $\Gamma$-equivariant $A$-bimodule if it
is both a $\Gamma$-module and $A$-bimodule, and the structure maps
$A\ox M\to M$ and $M\ox A\to M$ are maps of $\Gamma$-modules.
Morphisms of $\Gamma$-equivariant $A$-bimodules are maps which are
$A^e$-linear and $\Gamma$-linear independently.  The category of such
modules will be denoted $\mathrm{EQ}_{\Gamma} A^e$-mod.  We define
$\Gamma$-equivariant $A^e$-complexes similarly.
\label{equiv}
\end{definition}

To ease notation we may at times write ``equivariant bimodule''
instead of the full $\Gamma$-equivariant $A$-bimodule.  One example of
an equivariant bimodule is $A$ itself.  One can think of an
equivariant bimodule as an $A$-bimodule internal to the monoidal
category ($\Gamma$-mod,$\ox$).
\par
Kaygun has shown that the category $\mathrm{EQ}_{\Gamma} A^e$-mod is
actually the module category of a certain smash product $A^e\# \Gamma$
\cite[Lemma 3.3]{kaygun}.  Whence $\mathrm{EQ}_{\Gamma} A^e$-mod is
seen to be abelian with enough projectives.  Additionally,
$\mathrm{EQ}_{\Gamma} A^e$-mod comes equipped with restriction
functors (forgetful functors) to $A^e$-modules and $\Gamma$-modules.
Since $A^e\#\Gamma$ is free over both $A^e$ and $\Gamma$, one can
verify that these restriction functors preserve projectives.
\par
In this section we produce a projective bimodule resolution of
$A\#\Gamma$ via the smash product construction outlined below.

\begin{definition}
Let $X$ be any $\Gamma$-equivariant $A^e$-complex and let $Y$ be any
complex of Hopf bimodules.  The smash product complex $X\#Y$ is
defined to be the tensor complex $X\ox Y$ with the left and right $A\#\Gamma$-actions
\beq
a\cdot (x\ox y):=(ax)\ox y,\ \gamma\cdot (x\ox y):=(\gamma_1x)\ox (\gamma_2y),\ (x\ox y)\cdot a:=x(^{y_{-1}}a)\ox y_0
\eeq
and
\beq
(x\ox y)\cdot \gamma:=x\ox (y\gamma),
\eeq
for $x\in X$, $y\in Y$, $a\in A$ and $\gamma\in \Gamma$.
\label{defOne}
\end{definition}

Obviously, we can define the smash product of an equivariant bimodule
with a Hopf bimodule by considering them to be complexes concentrated
in degree 0.  The smash product construction is (bi)functorial in the
sense of the following

\begin{lemma}
If $f:X\to X'$ and $g:Y\to Y'$ are maps of complexes of
$\Gamma$-equivariant $A$-bimodules and complexes of Hopf bimodules
respectively, then the product map $f\ox g:X\# Y\to X'\# Y'$ is a map
of complexes of $A\#\Gamma$-bimodules.
\label{smashfunctr}
\end{lemma}
\begin{proof}
Left $A$-linearity of $f\ox g$ follows from left $A$-linearity of $f$
and right $\Gamma$-linearity follows from right $\Gamma$-linearity of
$g$.  Left $\Gamma$-linearity of $f\ox g$ follows from the fact that
both $f$ and $g$ are left $\Gamma$-linear.  Finally, right
$A$-linearity of $f\ox g$ follows from right $A$-linearity of $f$ and
$\Gamma$-colinearity of $g$.
\end{proof}
\par
Now, let $K$ be a projective resolution of $A$ as an $A$-bimodule,
with quasi-isomorphism $\tau:K\to A$.  We will assume that $K$ has the
following additional properties:

\begin{enumerate}[(I)]
\item there is a $\Gamma$-action on $K$ giving it the structure of a
  complex of $\Gamma$-equivariant $A$-bimodules, and the
  quasi-isomorphism $\tau: K\to A$ is $\Gamma$-equivariant.
\item $K$ is free over $A^e$ on a graded base space $\bar{K}\subset K$
  which is also a $\Gamma$-submodule.
\end{enumerate}

An example of a resolution of $A$ satisfying the above conditions is the bar resolution
\beq
BA=\cdots\to A\ox A^{\ox 2}\ox A\to A\ox A\ox A\to A\ox A\to 0,
\eeq
with its standard differential
\beqlbl
b\ox a_1\ox\dots \ox a_n\ox b'\mapsto 
\ba{l}
ba_1\ox\dots\ox b'+(-1)^n b\ox\dots\ox a_nb'\\
\hspace{5mm}+\sum_{i=1}^{n-1}(-1)^{i}b\ox\dots\ox a_ia_{i+1}\ox\dots\ox b'.
\ea
\label{bardif}
\eeqlbl
We give $BA$ the natural diagonal $\Gamma$-action
\beq
\gamma\cdot(b\ox a_1\ox\dots \ox a_n\ox b')=\gamma_1 b\ox \gamma_2a_1\ox\dots\ox \gamma_{n+1}a_n\ox \gamma_{n+2}b'.
\eeq
In this case, $\overline{BA}$ will be the graded subspace
$\overline{BA}=\bigoplus_n \k\ox A^{\ox n}\ox \k$.  One can also use
the reduced bar complex or, if $A$ is a Koszul algebra and $\Gamma$
acts by graded endomorphisms, we can take $K$ to be the Koszul resolution.
\par
For any $K$ satisfying (I) and (II), and any resolution $L$ of the
trivial right $\Gamma$-module $\k$, we can form the smash product
complex $K\# L^{\uparrow}$ using the coaction on $L^\uparrow$ defined
in the previous section.  We will see that the smash product complex
$K\# L^{\uparrow}$ provides a projective resolution of $A\#\Gamma$.

\begin{lemma}
Suppose $\tau:K\to A$ is a bimodule resolution of $A$ satisfying (I)
and (II), and let $\xi:L\to \k$ be any projective resolution of the
trivial right $\Gamma$-module.  Let $\xi^\uparrow:L^\uparrow\to
\Gamma$ be the quasi-isomorphism of Lemma \ref{setup}.  Then the
product map $\tau\ox \xi^\uparrow: K\# L^{\uparrow}\to A\#\Gamma$ is a
quasi-isomorphism of $(A\#\Gamma)^e$-complexes.
\label{lemmamma}
\end{lemma}
\begin{proof}
The fact that $\tau\ox \xi^{\uparrow}$ is a quasi-isomorphism follows
from the facts that both $\tau$ and $\xi^{\uparrow}$ are
quasi-isomorphisms, and that the tensor product of any two
quasi-isomoprhisms (over a field) is yet another quasi-isomorphism.
It is trivial to check that $\tau\ox \xi^{\uparrow}$ respects the left
$A$-action and right $\Gamma$-action.   One can simply use the
definition of the $A$ and $\Gamma$-actions on $K\# L^\uparrow$ given
in Definition \ref{defOne}.  For the left $\Gamma$-action we have, for
any $x\in K$, $l\in L^{\uparrow}$, and $\gamma\in \Gamma$,
\beq
\ba{rl}
(\tau\ox \xi^{\uparrow})(\gamma\cdot(x\ox l))&= \tau(\gamma_1x)\xi^{\uparrow}(\gamma_2l)\\
&={^{\gamma_1}\tau(x)}\gamma_2\xi^{\uparrow}(l)\\
&=\gamma(\tau(x)\xi^{\uparrow}(l))\\
&=\gamma\cdot \left((\tau\ox \xi^{\uparrow})(x\ox l)\right).
\ea
\eeq
So $\tau\ox \xi^{\uparrow}$ is a map of left $\Gamma$-complexes.
\par
Recall that, by Proposition \ref{propH}, the quasi-isomorphism
$\xi^\uparrow:L^\uparrow\to \Gamma$ is one of Hopf bimodules.  Hence
$l_{-1}\ox \xi^{\uparrow}(l_0)=\xi^{\uparrow}(l)_1\ox
\xi^{\uparrow}(l)_2$ for all $l\in L^{\uparrow}$.  Thus, for any $a\in
A$ we have
\beq
\ba{rlr}
(\tau\ox \xi^{\uparrow})((x\ox l)\cdot a)&= \tau(x({^{l_{-1}}a}))\xi^{\uparrow}(l_0)\\
&=\tau(x({^{\xi^\uparrow(l)_1}a}))\xi^{\uparrow}(l)_2 \\
&=\tau(x)({^{\xi^\uparrow(l)_1}a}\xi^{\uparrow}(l)_2) & (A\text{-linearity of $\tau$})\\
&=(\tau(x)\xi^{\uparrow}(l))a\\
&=\left((\tau\ox \xi^{\uparrow})(x\ox l)\right)a.
\ea
\eeq
This verifies that $\tau\ox \xi^{\uparrow}$ is a map of right
$A$-complexes and completes the proof that $\tau\ox \xi^{\uparrow}$ is
a quasi-isomorphism of $(A\#\Gamma)^e$-complexes.
\end{proof}

\begin{theorem}
Let $K$ be a bimodule resolution of $A$ satisfying conditions (I) and
(II), and let $L$ be any projective resolution of the trivial right
$\Gamma$-module $\k$.  Then the smash product $K\# L^{\uparrow}$ is a
projective $A\#\Gamma$-bimodule resolution of $A\#\Gamma$.
\label{lol}
\end{theorem}

The proof of the theorem will be clear from the following lemma.

\begin{lemma}
Let $M$ be a $\Gamma$-equivariant bimodule which is free on a base
space $\bar{M}\subset M$ satisfying $\Gamma \bar{M}=\bar{M}$, and
suppose $N$ is a projective right $\Gamma$-module.  Then the smash
product module $M\# N^\uparrow$ is projective over $A\#\Gamma$.
\label{lastlem}
\end{lemma}

\begin{proof}
Suppose that $N$ is free on some base $\bar{N}\subset N$.  Then we
have (bi)module isomorphisms $A\ox \bar{M}\ox A\cong M$ and
$\bar{N}\ox \Gamma\cong N$ given by restricting the action maps $A\ox
M\ox A\to M$ and $N\ox\Gamma\to N$.  We will also have $\Gamma\ox
\bar{N}\ox \Gamma\cong N^\uparrow$.  Now the restriction of the action
map on $M\# N\cong (A\ox\bar{M}\ox A)\#(\Gamma\ox \bar{N}\ox\Gamma)$
provides an isomorphism
\beq
A\#\Gamma\ox (\bar{M}\ox \bar{N})\ox A\#\Gamma\to M\# N^\uparrow
\eeq
with inverse
\beq
\ba{c}
M\# N^\uparrow\cong (A\ox \bar{M}\ox A)\#(\Gamma\ox \bar{N}\ox \Gamma)\to A\#\Gamma\ox (\bar{M}\ox \bar{N})\ox A\#\Gamma\\
(a\ox m\ox a')\ox(\gamma\ox n\ox \gamma')\mapsto a\gamma_2\ox(S^{-1}(\gamma_1)m\ox n)\ox {^{S(\gamma_3)}a'}\gamma'.
\ea
\eeq
So the smash product is a free $A\#\Gamma$-bimodule.
\par
In the case that $N$ is not free, we know that $N$ is a summand of
some free module $\mathcal{N}$.  This will imply that $N^\uparrow$ is
a summand of $\mathcal{N}^\uparrow$ as a Hopf bimodule.  It follows
that $M\# N^{\uparrow}$ is a summand of the free module
$M\#\mathcal{N}^\uparrow$, and hence projective.
\end{proof}

\begin{proof}[Proof of Theorem \ref{lol}]
We already know that there is a quasi-isomorphism of
$(A\#\Gamma)^e$-complexes $K\# L^{\uparrow}\to A\#\Gamma$, by Lemma
\ref{lemmamma}.  So we need only show that the smash product complex
is projective in each degree.  We have chosen $K$ so that each $K^i$
is an equivariant bimodule satisfying the hypotheses of Lemma
\ref{lastlem}, and each $L^j$ is projective by choice.  So each $K^i\#
(L^j)^\uparrow=K^i\#(L^\uparrow)^j$ is projective by Lemma
\ref{lastlem}.  Now, projectivity of the smash product $K\#
L^\uparrow$ in each degree follows from the fact that each $(K\#
L^\uparrow)^n$ is a finite sum of projective modules $K^i\#(L^\uparrow)^j$.
\end{proof}

\begin{remark}
The resolution of $A\#\Gamma$ constructed above is one of a number
resolutions that have appeared in the literature.  In \cite{GuGu},
Guccione and Guccione provide a resolution $X$ of the smash
$A\#\Gamma$ which is the tensor product of the bar resolution of $A$
with the bar resolution of $\Gamma$, along with some explicit
differential.  In the case that $\Gamma$ is a group algebra, Shepler
and Witherspoon have provided a class of resolutions of the smash
product \cite[Section 4]{SW12}.  Our resolution $K\# L^{\uparrow}$ is
a member of their class of resolutions (up to isomorphism).  The
reader should be aware that the construction given in \cite{SW12} is
somewhat different than the one given here.
\end{remark}

\section{Hochschild cochains as derived invariants}
\label{HCvDI}

Let $\Gamma$ be a Hopf algebra and $A$ be a $\Gamma$-module algebra.

\begin{definition}
Let $M$ be a complex of $A\#\Gamma$-bimodules and let $X$ be a complex
of $\Gamma$-equivariant $A$-bimodules.  We define a right
$\Gamma$-module structure on the set of homs $\Hom_{A^e}(X,M)$ by the formula
\beq
f\cdot\gamma (x):=S(\gamma_1)f(\gamma_2 x)\gamma_3,
\eeq
where $f\in \Hom_{A^e}(X,M)$, $\gamma\in\Gamma$, and $x\in X$.
\label{notez}
\end{definition}

This action was also considered in \cite{GuGu}, and similar actions
have appeared throughout the literature (see for example \cite[Section
5]{KKZ09}).  The first portion of the action, $S(\gamma_1)f(\gamma_2
x)$, assures that $f\cdot \gamma$ preserves left $A$-linearity.  The
additional right action is necessary to preserve right $A$-linearity.

\begin{lemma}
Let $M$ be a complex of $A\#\Gamma$-bimodules and $X$ be a complex of
$\Gamma$-equivariant $A$-bimodules.  The $\Gamma$-module structure on
$\Hom_{A^e}(X,M)$ given in Definition \ref{notez} is compatible with
the differential on the hom complex.  That is to say,
$\Hom_{A^e}(X,-)$ is a functor from $A\#\Gamma$-{\rm complexes} to $\Gamma$-{\rm complexes}.
\label{lemprem}
\end{lemma}

\begin{proof}
Recall that the differential on the hom complex is given by
$d:f\mapsto d_Mf\pm fd_X$.  So $\Gamma$-linearity of the differential
on the hom complex follows by $\Gamma$-linearity of $d_M$ and $d_X$.
\end{proof}

Let $L$ be a projective resolution of the trivial right
$\Gamma$-module $\k$, and $K$ be a bimodule resolution of $A$
satisfying conditions (I) and (II) of the previous section.   For a
complex of $A\#\Gamma$-bimodules $M$, any map $\theta\in \Hom_{\k}(K\#
L^{\uparrow}, M)$, and any $l\in L^{\uparrow}$, we let $\theta(-\ox
l)$ denote the $\k$-linear map
\beq
\ba{c}
K\to M\\
x\mapsto (-1)^{|x||l|}\theta((l_{-1}x)\ox l_{0}).
\ea
\eeq 
\par
Before giving the main theorem of this section let us highlight some
points of interest.  First, note that there is an embedding of chain
complexes $L\to L^\uparrow$, $\ell\mapsto \ell\ox_\Gamma 1$.  This map
becomes $\Gamma$-linear if we take the codomain to be $L^\uparrow$
with the adjoint action.  It is via this map that we view $L$ as a
subcomplex in $L^{\uparrow}$.  Second, note that for any $l\in
L\subset L^{\uparrow}$ we have $\rho(l)=1\ox l$.  Therefore, for all
$l\in L\subset L^{\uparrow}$,  $\theta(-\ox l)$ is just the map $x\mapsto (-1)^{|x||l|}\theta(x\ox l)$.

\begin{theorem}
Let $L$ be a projective resolution of the trivial right
$\Gamma$-module $\k$, and $K$ be a bimodule resolution of $A$
satisfying conditions (I) and (II).  Then for any complex $M$ of
$A\#\Gamma$-bimodules the map
\beq
\ba{c}
\Xi:\Hom_{(A\#\Gamma)^e}(K\# L^{\uparrow}, M)\to \Hom_\Gamma(L,\Hom_{A^e}(K,M))\\
\theta\mapsto (l\mapsto \theta(-\ox l))
\ea
\eeq
is a natural isomorphism of chain complexes.
\label{bigthm1}
\end{theorem}

In light of Theorem \ref{lol}, we are claiming that there is an
explicit natural isomorphism of derived functors
\beq
\RHom_{(A\#\Gamma)^e}(A\#\Gamma,-)\overset{\cong}\to \RHom_\Gamma(\k,\RHom_{A^e}(A,-)).
\eeq

\begin{proof}
To distinguish between the action of $\Gamma$ on $L$ as a subcomplex
in $L^{\uparrow}$, and the action of $\Gamma$ on $L$ itself, we will
denote the action of $\Gamma$ on $L^{\uparrow}$ by juxtaposition, and
the action of $\Gamma$ on $L$ by a dot $\cdot$.  So, for $l\in
L\subset L^{\uparrow}$ and $\gamma\in\Gamma$, we have
\beq
l\cdot\gamma=S(\gamma_1)l\gamma_2.
\eeq
It is straightforward to check that $\Xi$ is a map of chain complexes,
and we omit the computation.  We need to check that, for each
$\theta$, the map $\Xi(\theta)$ is a right $\Gamma$-linear, that each
$\theta(-\ox l)$ is $A^e$-linear, and that $\Xi$ is bijective.  
\par
Fix a homogeneous $A\#\Gamma$-bimodule map $\theta:K\# L^{\uparrow}\to
M$.  Since the coaction on $L^{\uparrow}$ restricts to a trivial
coaction on $L$, the map $\theta(-\ox l):K\to M$ is seen to be
$A^e$-linear for any $l\in L$.  Furthermore, for any $\gamma\in
\Gamma$, $l\in L$, and $x\in K$, $\Gamma$-linearity of $\theta$ on the
left and right gives the sequence of equalities
\beq
\ba{rl}
\theta(-\ox l\cdot \gamma)(x)&=(-1)^{|l||x|}\theta(x\ox l\cdot \gamma)\\
&= (-1)^{|l||x|}\theta (x\ox S(\gamma_1)l\gamma_2)\\
&=(-1)^{|l||x|}\theta((S(\gamma_2)\gamma_3 x)\ox S(\gamma_1)l\gamma_4)\\
&=(-1)^{|l||x|} S(\gamma_1)\theta((\gamma_2 x)\ox l)\gamma_3\\
&=(\theta(-\ox l)\cdot\gamma)(x).
\ea
\eeq
So we see that $\Xi(\theta)$ is in fact a right $\Gamma$-linear map $L\to \Hom_{A^e}(K,M)$. 
\par
To see that $\Xi$ is an isomorphism we provide an explicit inverse.
By a computation similar to (\ref{cmptn}), one can check that
$L^\uparrow$ is generated as a right $\Gamma$-complex by the
subcomplex $L\subset L^\uparrow$.  Using this fact, we define, for any
$\Gamma$-linear map
\beq
\chi: L\to \Hom_{A^e}(K,M),
\eeq
a graded vector space map $\Phi(\chi):K\# L^\uparrow\to M$.  For $l\in
L$, $\gamma\in \Gamma$, and $x\in K$ take
\beqlbl
\Phi(\chi)(x\ox l\gamma):= (-1)^{|x||l|}\chi(l)(x)\gamma.
\label{thetachi}
\eeqlbl
Let us assume for the moment that $\Phi(\chi)$ is well defined.  We
will return to this point at the end of the proof.
\par
The fact that $\Phi(\chi)$ is left $A$-linear and right
$\Gamma$-linear is clear.  Right $A$-linearity follows from right
$A$-linearity of $\chi(l)$ and the fact that coaction on $L^\uparrow$
restricts to a trivial coaction on $L$.  For left $\Gamma$-linearity,
let $x\in K$, $l\in L\subset L^\uparrow$, and $\gamma\in \Gamma$.  We have
\beq
\ba{rlr}
\Phi(\chi)(\gamma(x\ox l))& =\Phi(\chi)(\gamma_1x\ox \gamma_2l)\\
&=\Phi(\chi)\left(\gamma_1x\ox (l \cdot S^{-1}(\gamma_3))\gamma_2\right)\\
&=(-1)^{|x||l|}\chi(l \cdot S^{-1}(\gamma_3))(\gamma_1x)\gamma_2\\
&=(-1)^{|x||l|}\gamma_5\chi(l)(S^{-1}(\gamma_4)\gamma_1x)S^{-1}(\gamma_3)\gamma_2 &\text{($\Gamma$-linearity of $\chi$)}\\
&=(-1)^{|x||l|}\gamma_3\chi(l)(S^{-1}(\gamma_2)\gamma_1x)\\
&=(-1)^{|x||l|}\gamma\chi(l)(x)\\
&=\gamma\Phi(\chi)(x\ox l).
\ea
\eeq
We can use right $\Gamma$-linearity of $\Phi(\chi)$ to extend the
above computation to all of $K\# L^\uparrow=K\ox L\Gamma$.  Whence we
see that $\Phi(\chi)$ is a $A\#\Gamma$-bimodule map for arbitrary
$\chi:L\to \Hom_{A^e}(K,M)$.  The equalities
$\Phi(\Xi(\theta))=\theta$ and $\Xi(\Phi(\chi))=\chi$ follow by
construction.  So $\Phi=\Xi^{-1}$ and $\Xi=\Phi^{-1}$.
\par
Now, let us deal with the question of whether or not $\Phi(\chi)$ is
well defined.  In the case that $L$ is free on a subspace $\bar{L}$,
we will have $L^\uparrow=\Gamma\ox \bar{L}\ox \Gamma$.  We can then
define the $\k$-linear map $K\# L^{\uparrow}=K\ox \Gamma\ox \bar{L}\ox
\Gamma\to M$ on monomials by
\beq
x\ox\gamma\ox \bar{l}\ox \gamma'\mapsto x\ox \bar{l}\ox S^{-1}(\gamma_2)\ox \gamma_1\ox \gamma'\mapsto(-1)^{|x||\bar{l}|}\chi(\bar{l}\cdot S^{-1}(\gamma_2))(x)\gamma_1\gamma'.
\eeq
In the case that $\gamma\ox \gamma'=S(\gamma_1)\ox \gamma_2$, i.e. in
the case that $\gamma\ox \bar{l}\ox \gamma'$ is in $L\subset
L^{\uparrow}$, the image of this map is
$(-1)^{|x||\bar{l}|}\chi(\bar{l}\cdot \gamma)(x)$.  So we see that we
have recovered $\Phi(\chi)$ as defined at (\ref{thetachi}), and it
follows that $\Phi(\chi)$ is well defined.  We can deal with the
general case, in which $L$ is simply projective in each degree, by
noting that $L$ will be a summand of a free resolution.
\end{proof}

\begin{corollary}
Let $L$ and $K$ be as in Theorem \ref{bigthm1}, and $M$ be a
$A\#\Gamma$-bimodule.  Then we have a graded isomorphism
\beq
\mathrm{HH}(A\#\Gamma,M)\cong \mathrm{H}\big(\Hom_\Gamma(L,\Hom_{A^e}(K,M))\big).
\eeq
\end{corollary}
\begin{proof}
This follows from Theorem \ref{bigthm1} and the fact that the
Hochschild cohomology is given by the homology of the complex
$\Hom_{(A\#\Gamma)^e}(K\# L^{\uparrow}, M)$, since $K\# L^\uparrow$ is
a projective bimodule resolution of the smash product $A\#\Gamma$.
\end{proof}

We can, in fact, replace our resolution $K$ with any equivariant
$A^e$-projective resolution of $A$.  Let $P\to A$ be any
$\Gamma$-equivariant $A$-bimodule resolution of $A$ which is
projective over $A^e$.  By a straightforward process, we can produce
an equivariant complex $Q$ admitting equivariant quasi-isomorphisms
$P\to Q$ and $K\to Q$.  First, take $d^0$ to be the coproduct map
$K^0\oplus P^0\to A$.  Then we construct $Q$ inductively as the complex
\beq
Q=\cdots\to  K^2\oplus (A\ox\ker d^1\ox A)\oplus P^2\overset{d^2}\to K^1\oplus (A\ox\ker d^0\ox A)\oplus P^1\overset{d^1}\to K^0\oplus P^0\to 0,
\eeq
where $\Gamma$ acts diagonally on the summands $(A\ox\ker d^i\ox A)$.
Note that $Q$ is a complex of $\Gamma$-equivariant bimodules, and that
each $Q^i$ is projective over $A^e$.  The map $d^0:Q\to A$ is a
quasi-isomorphism by construction.  Whence we see that the two
inclusions $i_K:K\to Q$ and $i_P:P\to Q$ are equivariant
quasi-isomorphisms.  Taking
\beq
\mathscr{X}=\Hom_\Gamma(L,\Hom_{A^e}(Q,M))
\eeq
then gives the following corollary.

\begin{corollary}
Let $L$ and $K$ be as in Theorem \ref{bigthm1}, and $M$ be a complex
of $A\#\Gamma$-bimodules.  Let $P\to A$ be an equivariant bimodule
resolution of $A$ which is projective over $A^e$.  The complex
$\mathscr{X}$ admits quasi-isomorphisms 
\beq
\Hom_{(A\#\Gamma)^e}(K\# L^{\uparrow}, M)\overset{\sim}\leftarrow \mathscr{X}\overset{\sim}\to\Hom_\Gamma(L,\Hom_{A^e}(P,M)).
\eeq
\label{bigcor1}
\end{corollary}
\begin{proof}
Since $L$ is a bounded above complex of projectives, the functor
$\Hom_\Gamma(L,-)$ preserves quasi-isomorphisms.  Whence the proposed
quasi-isomorphisms can be given by
\beq
\Hom_\Gamma(L,\Hom_{A^e}(Q,M))\overset{(i_K^\ast)_\ast}\to \Hom_\Gamma(L,\Hom_{A^e}(K,M))\cong \Hom_{(A\#\Gamma)^e}(K\# L^{\uparrow}, M)
\eeq
and
\beq
\Hom_\Gamma(L,\Hom_{A^e}(Q,M))\overset{(i_P^\ast)_\ast}\to \Hom_\Gamma(L,\Hom_{A^e}(P,M)).
\eeq
\end{proof}

For $L$ and $K$ as above, and any $A\#\Gamma$-bimodule $M$, the
complex $\Hom_{\Gamma}(L,\Hom_{A^e}(K,M))$ is the total complex of the
first quadrant double complex
\beqlbl
{\tiny
\xymatrix{
 &\vdots & \vdots &\\
0\ar[r]& \Hom(L^{0},\Hom(K^{-1},M))\ar[r]^{{d_{L}}^\ast}\ar[u] & \Hom(L^{-1},\Hom(K^{-1},M))\ar[r]\ar[u] &\cdots\\
0\ar[r]& \Hom(L^{0},\Hom(K^{0},M))\ar[r]^{{d_{L}}^\ast}\ar[u]^{\pm(d_K^{\ast})_\ast} & \Hom(L^{-1},\Hom(K^{0},M))\ar[r]\ar[u]^{\pm(d_K^{\ast})_\ast} &\cdots\\
 &0\ar[u] & 0\ar[u] & . \\
}}
\label{bicmplx}
\eeqlbl
It follows that there are two spectral sequences converging to the
Hochschild cohomology of $A\#\Gamma$ with coefficients in $M$.
Filtering by the degree on $L$ produces a spectral sequence
\beq
E_2=\Ext_\Gamma(\k,\mathrm{HH}(A\#\Gamma, M)).
\eeq
\par
The existence of this spectral sequence is well known.  It first
appeared in the work of Stefan as a Grothendieck spectral sequence in
the setting of a Hopf Galois extension \cite{stefan}, and then in a
paper by Guccione and Guccione \cite[Corollary 3.2.3]{GuGu}.  Since
these results are well established, we do not elaborate on the details
here.  We will show in Section \ref{DIA} that both of these spectral
sequences can be used to calculate the cup product on Hochschild
cohomology when appropriate.  All necessary details will be given there.

\begin{notation}
The filtration induced by the degree of $L$ on the cohomology 
\beq
\mathrm{HH}(A\#\Gamma,M)=\mathrm{H}(\Hom_{\Gamma}(L,\Hom_{A^e}(K,M)))
\eeq
will be denoted $F^\Gamma$.  The filtration induced by the degree on
$K$ will be denoted $F^{A}$.  The associated graded spaces with
respect to these filtrations will be denoted
\beq
\mathrm{gr}_\Gamma \mathrm{HH}(A\#\Gamma,M)=\bigoplus_i \frac{F^\Gamma_i \big(\mathrm{HH}(A\#\Gamma,M)\big)}{F^\Gamma_{i-1}\big(\mathrm{HH}(A\#\Gamma,M)\big)}
\eeq
and
\beq
\mathrm{gr}_A \mathrm{HH}(A\#\Gamma,M)=\bigoplus_i \frac{F^A_i \big(\mathrm{HH}(A\#\Gamma,M)\big)}{F^A_{i-1}\big(\mathrm{HH}(A\#\Gamma,M)\big)}
\eeq
respectively.
\label{filtnotes}
\end{notation}

\section{Reminder of the cup products on Hochschild cohomology and derived invariant algebras}
\label{rmind}

The following general approach to the cup product on Hochschild
cohomology follows \cite{SiW}.  Let $R$ be any algebra and let $B$ be
an algebra extension of $R$, i.e. an algebra equipped with an algebra
map $R\to B$.  Let $P$ be a projective $R$-bimodule resolution of $R$
with quasi-isomorphism $\varphi:P\to R$.  Then $P\ox_RP$ is also a
projective resolution of $R$ with quasi-isomorphism
$\varphi\ox_R\varphi:P\ox_RP\to R$.  Whence there exists a
quasi-isomorphism $\omega:P\to P\ox_RP$ which fits into a diagram
\beqlbl
\xym{
P\ar[rr]^\omega\ar[dr]_\varphi & & P\ox_R P\ar[dl]^{\varphi\ox_R\varphi}\\
& R &
}
\label{dg}
\eeqlbl
and is unique up to homotopy.  From this we get a product map
\beq
\ba{c}
\Hom_{R^e}(P,B)\ox \Hom_{R^e}(P,B)\to \Hom_{R^e}(P,B)\\
f\ox g\mapsto \mu_B(f\ox_R g)\omega,
\ea
\eeq
and subsequent dg algebra structure on $\Hom_{R^e}(P,B)$.  One can
check that any choice of $\omega$ results in the same product on the
cohomology $\mathrm{HH}(R,B)$.  We call this product the \emph{cup
  product}.  Note that the dg algebra $\Hom_{R^e}(P,B)$ need not be
associative, but it will be associative up to a homotopy.
\par
Suppose now that $\Gamma$ is a Hopf algebra and $L$ is a projective
resolution of $\k_\Gamma=\Gamma/\ker\epsilon$.  Let $\mathscr{B}$ be a
right $\Gamma$-module dg algebra.  (We do not require that
$\mathscr{B}$ is strictly associative.)  Since $\Gamma\ox\Gamma$ is
free over $\Gamma$, the diagonal action on $L\ox L$ makes it into a
projective resolution of $\k$ as well.  So, again, we have a
quasi-isomorphism $\sigma:L\to L\ox L$ which is unique up to homotopy
and fits into a diagram analogous to (\ref{dg}).  Hence, we get a
similarly defined product on the derived invariants
\beq
\ba{c}
\Hom_{\Gamma}(L,\mathscr{B})\ox \Hom_{\Gamma}(L,\mathscr{B})\to \Hom_{\Gamma}(L,\mathscr{B})\\
f\ox g\mapsto \mu_{\mathscr{B}}(f\ox g)\sigma.
\ea
\eeq
This product is unique on cohomology and gives
$\Hom_{\Gamma}(L,\mathscr{B})$ the structure of a
(not-necessarily-associative) dg algebra.

\section{Hochschild cohomology as a derived invariant algebra}
\label{DIA}

Let $\Gamma$ be a Hopf algebra and $A$ be a $\Gamma$-module algebra.
We also fix a bimodule resolution $\tau:K\to A$ which satisfies
conditions (I) and (II) of Section \ref{bimodrez}, and a projective
resolution $\xi:L\to \k$ of the trivial right $\Gamma$-module.  From
here on out we assume $K$ also satisfies
\begin{enumerate}
\item[(III)] there is a quasi-isomorphism $\omega:K\to K\ox_AK$ of
  complexes of $\Gamma$-equivariant $A$-bimodules.
\end{enumerate}
As was stated in the previous section, there will always be some
quasi-isomorphism $\omega$ of $A^e$-complexes.  The content of
condition (III) is that we may choose $\omega$ to be $\Gamma$-linear.  
\par
In the case of the bar resolution
\beq
BA=\cdots\to A\ox A^{\ox 2}\ox A\to A\ox A\ox A\to A\ox A\to 0
\eeq
the map $\omega$ is given by
\beqlbl
\omega: b\ox a_1\ox\dots\ox a_n\ox b'\mapsto \sum_{0\leq i\leq n} (b\ox a_1\ox\dots \ox a_i\ox 1)\ox_A(1\ox a_{i+1}\ox\dots\ox a_n\ox b')
\label{comultbar}
\eeqlbl
We will denote the image of $\omega$ using a Sweedler's type notation,
as if $\omega$ were a comultiplication.  Specifically, on elements we
take $\omega_1(x)\ox_A \omega_2(x)=\omega(x)$, with the sum
suppressed.  In this notation, $\Gamma$-linearity of $\omega$ is
equivalent to the equality $\omega_1(\gamma x)\ox_A \omega_2(\gamma
x)=\gamma_1\omega_1(x)\ox_A\gamma_2\omega_2(x)$, for all $\gamma\in
\Gamma$ and $x\in K$.
\par
Let us also fix a quasi-isomorphism $\sigma:L\to L\ox L$.  As with
$\omega$ and $K$, we denote the image of $l\in L$ under $\sigma$ by
$\sigma_1(l)\ox\sigma_2(l)$.  In this notation $\Gamma$-linearity
appears as $\sigma_1(l\cdot \gamma)\ox\sigma_2(l\cdot \gamma)=\sigma_1(l)\cdot\gamma_1\ox\sigma_2(l)\cdot\gamma_2$.

\begin{proposition}
For any algebra extension $B$ of $A\#\Gamma$, the complex
$\Hom_{A^e}(K,B)$, with the product of Section \ref{rmind} and
$\Gamma$-action of Definition \ref{notez}, is a right $\Gamma$-module
dg algebra.
\end{proposition}
\begin{proof}
Let us denote the multiplication on $\Hom_{A^e}(K,B)$ by
juxtaposition.  We need to show that for functions $f,g\in
\Hom_{A^e}(K,M)$, and $\gamma\in \Gamma$, the formula
$(fg)\cdot\gamma=(f\cdot\gamma_1)(g\cdot\gamma_2)$ holds.  Let us
simply check on elements.  We have, for any $x\in K$,
\beq
\ba{rlr}
((fg)\cdot\gamma)(x)&= S(\gamma_1)(fg)(\gamma_2 x)\gamma_3\\
&=\pm S(\gamma_1)f(\gamma_2\omega_1(x))g(\gamma_3\omega_2(x))\gamma_4 & (\text{by }\Gamma\text{-linearity of }\omega)\\
&=\pm\big(S(\gamma_1)f(\gamma_2\omega_1(x))\gamma_3\big)\big(S(\gamma_4)g(\gamma_5\omega_2(x))\gamma_6\big)\\
&=(f\cdot \gamma_1)(\omega_1(x))(g\cdot \gamma_2)(\omega_2(x))\\
&=\big((f\cdot\gamma_1)(g\cdot\gamma_2)\big)(x).
\ea
\eeq
\end{proof}

According to this proposition, and the material of Section
\ref{rmind}, the double complex $\Hom_\Gamma(L,\Hom_{A^e}(K,B))$ will
now cary a natural dg algebra structure.
\par
We now seek to extend the diagonal map $\sigma$ on $L$ to a diagonal
map on the induced complex $L^\uparrow$.  One can verify that the
obvious map $L\ox L\to L^{\uparrow}\ox_\Gamma L^{\uparrow}$ is an
embedding, since the statement holds when $L$ is free.  In this way we
view $L\ox L$ as a subcomplex of $L^{\uparrow}\ox_\Gamma
L^{\uparrow}$.  The complex $L^\uparrow\ox_\Gamma L^\uparrow$ is taken
to be a $\Gamma$-comodule under the standard tensor $\Gamma$-comodule
structure $l\ox_\Gamma l'\mapsto (l_{-1}l'_{-1})\ox(l_0\ox_\Gamma
l'_0)$.  Since $L^\uparrow$ is itself a Hopf bimodule over $\Gamma$,
this coaction gives $L^{\uparrow}\ox_\Gamma L^{\uparrow}$ the
structure of a Hopf bimodule as well.  Before giving the next result
we also note that, on elements, commutativity of the diagram
\beq
\xym{
L\ar[rr]^{\sigma}\ar[dr]_\xi & & L\ox L\ar[dl]^{\xi\ox\xi}\\
 & \k &
}
\eeq
produces the equality $\xi(l)=\xi(\sigma_1(l))\xi(\sigma_2(l))$ for
each $l\in L$.

\begin{lemma}
The map $\sigma:L\to L\ox L\subset L^{\uparrow}\ox_\Gamma
L^{\uparrow}$ extends uniquely to a quasi-isomorphism of chain
complexes of Hopf-bimodules $\sigma^{\uparrow}:L^{\uparrow}\to L^{\uparrow}\ox_\Gamma L^{\uparrow}$.
\label{UP}
\end{lemma}
\begin{proof}
Let $\cdot$ denote the right action of $\Gamma$ on $L$ and
juxtaposition denote the action of $\Gamma$ on the bimodule
$L^\uparrow$.  Take $l\in L\subset L^\uparrow$ and
$\gamma,\gamma'\in\Gamma$.  We extend $\sigma$ to all of $L^\uparrow$
according to the formula
\beq
\sigma^{\uparrow}(\gamma l\gamma'):=\gamma\sigma_1(l)\ox_\Gamma\sigma_2(l)\gamma'.
\eeq
Recall that $l\cdot \gamma=S(\gamma_1)l\gamma_2$.  To see that
$\sigma^\uparrow$ is well defined, we first define the
$\Gamma^e$-linear map 
\beq
\Sigma:L\ox \Gamma^e=\Gamma\ox L\ox \Gamma\to L^\uparrow\ox_\Gamma L^\uparrow
\eeq
by the same formula $\Sigma(\gamma \ox
l\ox\gamma')=\gamma\sigma_1(l)\ox_\Gamma\sigma_2(l)\gamma'$.  The computation 
\beq
\ba{rlr}
\Sigma(S(\gamma_1)\ox l\ox \gamma_2)&= S(\gamma_1)\sigma_1(l)\ox_\Gamma\sigma_2(l)\gamma_2\\
&= S(\gamma_1)\sigma_1(l)\gamma_2\ox_\Gamma S(\gamma_3)\sigma_2(l)\gamma_4\\
&=\big(\sigma_1(l)\cdot \gamma_1\big)\ox_\Gamma \big(\sigma_2(l)\cdot\gamma_2\big)\\
&=\sigma_1(l\cdot \gamma)\ox_\Gamma \sigma_2(l\cdot\gamma) &\text{(by $\Gamma$-linearity of $\sigma$)}\\
&=\Sigma(1\ox (l\cdot \gamma)\ox 1)
\ea
\eeq
shows that $\sigma^{\uparrow}$ respects the necessary relations to
induce a map on the quotient $L^\uparrow=L\ox_\Gamma \Gamma^e$.  This
recovers our original map $\sigma^\uparrow$, and shows that it is in
fact well defined.
\par
Recall that $\sigma:L\to L\ox L$ was chosen so that
$\xi(l)=\xi(\sigma_1(l))\xi(\sigma_2(l))$ for all $l\in L$, and that
$\xi^{\uparrow}:L^\uparrow\to \Gamma$ is defined by $\gamma
l\gamma'\mapsto \gamma \xi(l)\gamma'$.  So we will have the
commutative diagram
\beq
\xym{
L^{\uparrow}\ar[rr]^{\sigma^{\uparrow}}\ar[dr]_{\xi^{\uparrow}} & & L^{\uparrow}\ox_\Gamma L^{\uparrow}\ar[dl]^{\xi^{\uparrow}\ox_\Gamma \xi^{\uparrow}}\\
 & \Gamma &.}
\eeq
The fact that $\sigma^\uparrow$ is a quasi-isomorphism follows from
commutativity of the above diagram and the fact that $\xi^\uparrow$
and $\xi^\uparrow\ox_\Gamma \xi^\uparrow$ are quasi-isomorphisms.  As
for colinearity of $\sigma^\uparrow$, by the definition of the
coaction on the induced complex $L^\uparrow$ given in Definition
\ref{codef}, we have
\beq
\ba{rl}
(\sigma^{\uparrow}(\gamma l\gamma'))_{-1}\ox (\sigma^{\uparrow}(\gamma l\gamma'))_0&=\big(\gamma\sigma_1(l)\ox \sigma_2(l)\gamma'\big)_{-1}\ox \big(\gamma\sigma_1(l)\ox \sigma_2(l)\gamma'\big)_0\\
&=(\gamma_1\gamma'_1)\ox(\gamma_2\sigma_1(l)\ox\sigma_2(l)\gamma'_2)\\
&=(\gamma_1\gamma'_1)\ox \sigma^{\uparrow}(\gamma_2l\gamma_2')\\
&=(\gamma l\gamma')_{-1}\ox \sigma^{\uparrow}((\gamma l\gamma)_0).
\ea
\eeq
\end{proof}

Now we have a quasi-isomorphism $\omega:K\to K\ox_A K$ and have
produced a quasi-isomorphism $\sigma^\uparrow:L^\uparrow\to
L^\uparrow\ox_\Gamma L^\uparrow$ from the given map $\sigma:L\to L\ox
L$.  We would like to use this information, along with some twisting,
to produce an explicit quasi-isomorphism
\beq
K\# L^\uparrow\to (K\# L^\uparrow)\ox_{A\#\Gamma}(K\# L^\uparrow).
\eeq
The next lemma offers the ``twisting'' portion of the proposed construction.

\begin{lemma}
The isomorphism of $\k$-complexes
\beq
\ba{c}
(K\ox_AK)\ox(L\ox L)\to (K\ox L)\ox_A(K\ox L)\\
(x\ox_Ay)\ox (l\ox l')\mapsto (-1)^{|l||y|} (x\ox l)\ox_{A}(y\ox l')
\ea
\eeq
extends uniquely to an isomorphism $\phi:( K\ox_A K)\#
(L^{\uparrow}\ox_\Gamma L^{\uparrow})\to (K\#
L^{\uparrow})\ox_{A\#\Gamma}(K\# L^{\uparrow})$ of complexes of $A\#\Gamma$-bimodules.
\label{diglem}
\end{lemma}
\begin{proof}
The map $\phi$ is given by
\beq
\phi:(x\ox_A y)\ox(l\ox_\Gamma l')\mapsto (-1)^{|l||y|}(x\ox l_0)\ox_{A\#\Gamma}(S^{-1}(l_{-1})y\ox l'),
\eeq
for $x,y\in K$, $l,l'\in L^\uparrow$.  The fact that $\phi$ is well
defined follows by standard manipulations, which we do not reproduce
here.  The fact that $\phi$ is a chain map can be verified by using
the $\Gamma$-linearity and $\Gamma$-colinearity of the differentials
on $K$ and $L^\uparrow$ respectively.
\par
In order to show that $\phi$ is an $A\#\Gamma$-bimodule map, the only
non-trivial things to check are left $\Gamma$-linearity and right
$A$-linearity.  For left $\Gamma$-linearity we have, for any $\gamma\in\Gamma$,
\beq
\ba{rl}
\phi(\gamma((x\ox_A y)\ox(l\ox_\Gamma l')))&=\phi((\gamma_1x\ox_A \gamma_2y)\ox(\gamma_3l\ox_\Gamma l'))\\
&=\pm(\gamma_1x\ox \gamma_4l_0)\ox_{A\#\Gamma}(S^{-1}(\gamma_3l_{-1})\gamma_2y\ox l')\\
&=\pm(\gamma_1x\ox \gamma_4l_0)\ox_{A\#\Gamma}(S^{-1}(l_{-1})S^{-1}(\gamma_3)\gamma_2y\ox l')\\
&=\pm(\gamma_1x\ox \gamma_2l_0)\ox_{A\#\Gamma}(S^{-1}(l_{-1})y\ox l')\\
&=\pm\gamma\left((x\ox l_0)\ox_{A\#\Gamma}(S^{-1}(l_{-1})y\ox l')\right)\\
&=\gamma\phi((x\ox_A y)\ox(l\ox_\Gamma l')).
\ea
\eeq
For right $A$-linearity we have, for any $a\in A$,
\beq
\ba{rl}
\phi(((x\ox_A y)\ox(l\ox_\Gamma l'))a)&=\phi((x\ox_A y({^{l_{-1}l'_{-1}}a}))\ox(l_0\ox_\Gamma l'_0)\\
& =\pm(x\ox l_0)\ox_{A\#\Gamma}(S^{-1}(l_{-1})(y({^{l_{-2}l'_{-1}}a}))\ox l'_0)\\
&=\pm(x\ox l_0)\ox_{A\#\Gamma}((S^{-1}(l_{-1})y)({^{S^{-1}(l_{-2})l_{-3}l'_{-1}}a}))\ox l'_0)\\
&=\pm(x\ox l_0)\ox_{A\#\Gamma}((S^{-1}(l_{-1})y)({^{l'_{-1}}a}))\ox l'_0)\\
&=\pm \left((x\ox l_0)\ox_{A\#\Gamma}(S^{-1}(l_{-1})y\ox l'_0)\right)a\\
&=\phi((x\ox_A y)\ox(l\ox_\Gamma l'))a.
\ea
\eeq
The inverse to $\phi$ is the map
\beq
(x\ox l)\ox_{A\#\Gamma}(y\ox l')\mapsto (-1)^{|l||y|}(x\ox_A l_{-1}y)\ox(l_{0}\ox_\Gamma l').
\eeq
\end{proof}

\begin{proposition}
Let $\phi$ be the isomorphism of Lemma \ref{diglem}, and
$\sigma^{\uparrow}:L^{\uparrow}\to L^{\uparrow}\ox_\Gamma
L^{\uparrow}$ be the quasi-isomorphism of Lemma \ref{UP}.  Then the map
\beq
\phi(\omega\ox \sigma^{\uparrow}):K\# L^{\uparrow}\to K\# L^{\uparrow}\ox_{A\#\Gamma}K\# L^{\uparrow}
\eeq
is a quasi-isomorphism of $(A\#\Gamma)^e$-complexes.
\label{digggg}
\end{proposition}

\begin{proof}
Since $\sigma^{\uparrow}$ is a map of complexes of Hopf bimodules by
Lemma \ref{UP}, and $\omega$ is a map of complexes of equivariant
$A$-bimodules by choice, $\omega\ox\sigma^{\uparrow}$ is
$(A\#\Gamma)^e$-linear, by Lemma \ref{smashfunctr}.  Also, the product
map $\omega\ox \sigma^\uparrow$ is a quasi-isomorphism since $\omega$
and $\sigma^\uparrow$ are themselves quasi-isomorphisms.  Now, since
$\phi$ is an isomorphism of $(A\#\Gamma)^e$-complexes, the claim follows.
\end{proof}

For any algebra extension $A\#\Gamma\to B$, we define the cup product
on $\Hom_{(A\#\Gamma)^e}(K\# L^{\uparrow}, B)$ by way of the diagonal
map $K\# L^{\uparrow}\to K\# L^{\uparrow}\ox_{A\#\Gamma}K\#
L^{\uparrow}$ given in Proposition \ref{digggg}.
\par
The following are the {\bf hypotheses for Theorem \ref{bigthm2}}: $K$
is a bimodule resolution of $A$ equipped with a diagonal map
$\omega:K\to K\ox_A K$ satisfying conditions (I)-(III), and $L$ is a
projective resolution of the trivial right $\Gamma$-module $\k$ with a
quasi-isomorphism $\sigma:L\to L\ox L$.  We give $K\# L^\uparrow$ the
diagonal quasi-isomorphism of Proposition \ref{digggg}. 

\begin{theorem}
For any algebra extension $B$ of the smash product $A\#\Gamma$, the isomorphism
\beq
\Xi:\Hom_{(A\#\Gamma)^e}(K\# L^{\uparrow}, B)\overset{\cong}\to \Hom_\Gamma(L,\Hom_{A^e}(K,B))
\eeq
of Theorem \ref{bigthm1} is one of (not-necessarily-associative) dg algebras.
\label{bigthm2}
\end{theorem}

Let us note that, if $K$ and $L$ are chosen appropriately, the dg
algebra $\Hom_\Gamma(L,\Hom_{A^e}(K,B))$ will be associative.  It
follows, by the theorem, that $\Hom_{(A\#\Gamma)^e}(K\# L^{\uparrow},
B)$ will also be associative in this case.  For example, one can
always take $K$ to be the bar resolution $BA$ of $A$ and $L$ to be the
bar resolution $\k\ox_\Gamma B\Gamma$ of $\k$ to get this property.

\begin{proof}
We want to verify commutativity of the diagram
\beqlbl
{\Small
\xymatrix{
\Hom(K\# L^{\uparrow}, B)\ox \Hom(K\# L^{\uparrow}, B)\ar[r]^(.43){\Xi\ox\Xi}\ar[d]^{mult} & \Hom(L,\Hom(K,B))\ox \Hom_\Gamma(L,\Hom(K,B))\ar[d]^{mult}\\
\Hom(K\# L^{\uparrow}, B)\ar[r]^{\Xi} & \Hom(L,\Hom(K,B)).
}}
\label{dig}
\eeqlbl
There are three multiplications we need to deal with here.  For the
purpose of this proof we will denote the products on 
\beq
\Hom_{(A\#\Gamma)^e}(K\# L^{\uparrow}, B),\ \Hom(L,\Hom_{A^e}(K,B)) ,\ \text{and}\ \Hom_{A^e}(K,B) 
\eeq
by a dot $\cdot$, an asterisk $\ast$, and juxtaposition respectively.
Let $\theta$ and $\theta'$ be functions in $\Hom_{(A\#\Gamma)^e}(K\#
L^\uparrow, A\#\Gamma)$ and fix arbitrary $x\in K$ and $l\in L\subset L^\uparrow$.
\par
Following around the top of (\ref{dig}) sends $\theta\ox\theta'$ to
the function $\Xi(\theta)\ast\Xi(\theta')\in
\Hom_\Gamma(L,\Hom_{A^e}(K,B))$.  This function sends $l\in L$ to the map 
\beq
(-1)^{|\theta'||\sigma_1(l)|}\theta(-\ox \sigma_1(l))\theta'(-\ox\sigma_2(l))
\eeq
 in $\Hom_{A^e}(K,B)$, where $\theta(-\ox \sigma_1(l))$ and
 $\theta'(-\ox\sigma_2(l))$ are as defined in the paragraphs preceding
 Theorem \ref{bigthm1}.  Since the coaction on $L^\uparrow$ restricts
 to the trivial coaction on $L$, the above function evaluated at $x$
 is the element
\beq
(-1)^{\epsilon}\theta\left(\omega_1(x)\ox \sigma_1(l)\right)\theta'\left(\omega_2(x)\ox\sigma_2(l)\right)\in B,
\eeq
where
\beq
\ba{rl}
\epsilon&=|\theta'||\sigma_1(l)|+|\omega_1(x)|(|\theta'|+|\sigma_2(l)|+|\sigma_1(l)|)+|\omega_2(x)||\sigma_2(l)|\\
&=|\theta'||\sigma_1(l)|+|\omega_1(x)|(|\theta'|+|l|)+|\omega_2(x)||\sigma_2(l)|.\\
&=|\theta'|(|\sigma_1(l)|+|\omega_1(x)|)+|\omega_1(x)||l|+|\omega_2(x)||\sigma_2(l)|.
\ea
\eeq
\par
Following around the bottom row sends $\theta\ox\theta'$ to the
function $\Xi(\theta\cdot\theta')\in \Hom_\Gamma(L,\Hom_{A^e}(K,B))$,
which takes our element $l\in L$ to $(\theta\cdot\theta')(-\ox l)\in
\Hom_{A^e}(K, B)$.  Evaluating at $x\in K$ produces the element
$(-1)^{|x||l|}(\theta\cdot\theta')(x\ox l)\in B$.  Recalling the
diagonal map on $K\# L^{\uparrow}$ given in Proposition \ref{digggg},
the formula for $\phi|(K\ox_AK)\ox(L\ox L)$ given in Lemma
\ref{diglem}, and the fact that $\sigma^\uparrow|L=\sigma$, we have
the equality
\beq
(-1)^{|x||l|}(\theta\cdot\theta')(x\ox l) =(-1)^{\epsilon'}\theta\left(\omega_1(x)\ox \sigma_1(l)\right)\theta'\left(\omega_2(x)\ox \sigma_2(l)\right),
\eeq
where
\beq
\ba{rl}
\epsilon'&=|x||l|+|\sigma_1(l)||\omega_2(x)|+|\theta'|(|\omega_1(x)|+|\sigma_1(l)|)\\
&=(|\omega_1(x)|+|\omega_2(x)|)|l|+|\sigma_1(l)||\omega_2(x)|+|\theta'|(|\omega_1(x)|+|\sigma_1(l)|)\\
&=|\omega_1(x)||l|+|\omega_2(x)||l|+|\omega_2(x)||\sigma_1(l)|+|\theta'|(|\omega_1(x)|+|\sigma_1(l)|)\\
&=|\omega_1(x)||l|+|\omega_2(x)|(|\sigma_1(l)|+|\sigma_2(l)|)+|\omega_2(x)||\sigma_1(l)|+|\theta'|(|\omega_1(x)|+|\sigma_1(l)|)\\
&\equiv |\omega_1(x)||l|+|\omega_2(x)||\sigma_2(l)|+|\theta'|(|\omega_1(x)|+|\sigma_1(l)|) \mod 2\\
&=\epsilon.
\ea
\eeq
So following around the top or bottom of (\ref{dig}) produces the same
function.
\end{proof}

\begin{corollary}
Let $B$ be an algebra extension of $A\#\Gamma$, and take $K$ and $L$
as in Theorem \ref{bigthm2}.  Then there is an isomorphism of algebras
\beq
\mathrm{HH}(A\#\Gamma, B)\cong\mathrm{H}\left(\Hom_\Gamma(L,\Hom_{A^e}(K,B))\right).
\eeq
\label{dimalg1}
\end{corollary}
\begin{proof}
This is an immediate consequence of Theorem \ref{bigthm2} and the fact
that
\beq
\mathrm{HH}(A\#\Gamma, B)=\mathrm{H}\big(\Hom_{(A\#\Gamma)^e}(K\# L^{\uparrow}, M)\big)
\eeq
as an algebra.
\end{proof}

As was the case with Theorem \ref{bigthm1}, we can drop condition (II)
on $K$.

\begin{corollary}
Let $K$, $L$, and $B$ be as in Theorem \ref{bigthm2}.  Let $P\to A$ be
a $\Gamma$-equivariant bimodule resolution of $A$ which is projective
over $A^e$ in each degree.  Suppose additionally that $P$ admits a
diagonal quasi-isomorphism $P\to P\ox_A P$ which is $\Gamma$-linear.
Then there is third dg algebra $\mathscr{A}$ which admits quasi-isomorphisms
\beqlbl
\Hom_{(A\#\Gamma)^e}(K\# L^{\uparrow}, B)\overset{\sim}\rightarrow \mathscr{A}\overset{\sim}\leftarrow\Hom_\Gamma(L,\Hom_{A^e}(P,B))
\label{mapmap}
\eeqlbl
which are all algebra maps up to a homotopy.
\label{bigcor2}
\end{corollary}

Let $\eta$ denote any of the maps of (\ref{mapmap}).  The main point
is that for any cycles $f$ and $g$ in the domain, the difference
$\eta(fg)-\eta(f)\eta(g)$ will be a boundary.  So all of the maps of
(\ref{mapmap}) become algebra isomorphisms on homology.  The proof of
this result is a bit of distraction, and has been relegated to the appendix.
\par
As was discussed in the introduction, a spectral sequence $E$ is
called multiplicative if it comes equipped with a bigraded products
$E^{pq}_r\ox E^{p'q'}_r\to E^{(p+p')(q+q')}_r$, for each $r$, such
that each differential $d_r:E_r\to E_r$ is a graded derivation and
each isomorphism $E_{r+1}\cong H(E_r)$ is one of algebras.  The
spectral sequence associated to any filtered dg algebra will be
multiplicative, for example.  For any multiplicative spectral sequence
$E$, the limiting term $E_\infty$ has the natural structure of a
bigraded algebra \cite[Multiplicative Structures 5.4.8]{weibel}.  We
say that a multiplicative spectral sequence converges to a graded
algebra $\mathrm{H}$ if $\mathrm{H}$ carries an additional filtration
and we have an isomorphism of bigraded algebras $E_\infty=\mathrm{gr} \mathrm{H}$.
\par
Recall the filtrations $F^A$ and $F^\Gamma$ on
$\mathrm{HH}(A\#\Gamma,B)$ given in Notation \ref{filtnotes}.  Since
the multiplication on the double complex
$\Hom_\Gamma(L,\Hom_{A^e}(K,B))$ is bigraded, both the row and column
filtrations (i.e. the filtrations induced by the degrees on $K$ and
$L$) give it the structure of a filtered dg algebra.  It follows that
both of the associated spectral sequences are multiplicative.  It also
follows that $F^A$ and $F^\Gamma$ are algebra filtrations on the
Hochschild cohomology.

\begin{corollary}
For any algebra extension $B$ of the smash product $A\#\Gamma$, there
are two multiplicative spectral sequences
\beq
E_2=\Ext_{\Gamma}(\k,\mathrm{HH}(A,B))\Rightarrow \mathrm{HH}(A\#\Gamma,B)
\eeq
and
\beq
'E_1=\Ext_{\Gamma}(\k,\RHom_{A\text{-}\mathrm{bimod}}(A,B))\Rightarrow \mathrm{HH}(A\#\Gamma,B)
\eeq
which converge to the Hochschild cohomology as an algebra.
\label{spectrill}
\end{corollary}
\begin{proof}
These spectral sequences are induced by the row and column filtrations
on the (first quadrant) double complex
$\Hom_\Gamma(L,\Hom_{A^e}(K,B))$.  Since the product on
$\Hom_\Gamma(L,\Hom_{A^e}(K,B))$ respects both of these filtrations,
and its homology is the Hochschild cohomology ring
$\mathrm{HH}(A\#\Gamma, B)$, both of the spectral sequences are
multiplicative and converge to the Hochschild cohomology.
\par
Filtering by the degree on $K$ produces the spectral sequence with
$'E_1=\Ext_{\Gamma}(\k,\RHom_{A^e}(A,B))$, where we take
$\RHom_{A^e}(A,B)=\Hom_{A^e}(K,B)$.  Filtering by the degree on $L$
produces a spectral sequence with
$E_1=\Hom_\Gamma(L,\mathrm{HH}(A,B))$, since each $\Hom_\Gamma(L^i,-)$
is exact and hence commutes with homology.  Since the differentials on
$E_1$ are given by $d_L^\ast$, it follows that the $E_2$ term is as described.
\end{proof}

In the language of Notation \ref{filtnotes}, the $E_\infty$-terms of
these spectral sequences are the bigraded algebras
$\gr_\Gamma\mathrm{HH}(A\#\Gamma,B)$ and
$\gr_A\mathrm{HH}(A\#\Gamma,B)$ respectively.

\begin{corollary}
If the global dimension of $\Gamma$ is $\leq 1$ then we have an
isomorphism of algebras 
\beq
\mathrm{gr}_\Gamma\mathrm{HH}(A\#\Gamma,B)\cong\Ext_{\Gamma}(\k,\mathrm{HH}(A,B)).
\eeq
\label{spectrilll}
\end{corollary}
\begin{proof}
In this case $E_2=E_\infty$.
\end{proof}

\section{Algebras of extensions as derived invariant algebras}
\label{coHOM!}

The main purpose of this section is to give some multiplicative
spectral sequences converging to the algebra $\Ext_{A\#\Gamma}(M,M)$,
for any $A\#\Gamma$-module $M$.  As was the case with Hochschild
cohomology, we will derive our spectral sequences from some explicit
isomorphism at the level of cochains.  Some related spectral sequences
for groups of extensions, without multiplicative structures, were
given in \cite[Section 3.2.6]{GuGu}.
\par
For any algebra $R$ and $R$-modules $M$ and $N$, there is a natural
bimodule structure on $\Hom_\k(M,N)$ induced by the left $R$-actions
on $M$ and $N$.  In the case that $M=N$, the bimodule structure is
induced by the associated representation $R\to \End_\k(M)$.  Since the
representation $R\to \End_\k(M)$ is an algebra map, the endomorphism
algebra of any $R$-module has the structure of an algebra extension of
$R$.
\par
The Yoneda product on $\Ext_R(M,M)$ is defined in the following
manner: first take a projective resolution $Q\to M$, then define
$\Ext_R(M,M)$ as the homology algebra of the endomorphism dg algebra
$\End_R(Q)$.  An alternate approach to the Yoneda product will be
given in Corollary \ref{cor15}.  Recall that, for any bimodule
resolution $P$ of $R$, the tensor product $P\ox_R M$ provides a
projective resolution of $M$.  This fact is a consequence of
K\"{u}nneth's spectral sequence.

\begin{proposition}[{\cite[Lemma 9.1.9]{weibel}}]
Let $R$ be any algebra, $P$ be a projective $R$-bimodule resolution of
$R$, and $M$ and $N$ be any modules over $R$.  The $\ox$-$\Hom$
adjunction gives an isomorphism of complexes
\beq
\Hom_{R^e}(P, \Hom_\k(M,N))\overset{\cong}\to \Hom_{R}(P\ox_RM,N).
\eeq
In the case that $M=N$, and $P$ is the bar resolution $P=BR$, there is
a quasi-isomorphisms of dg algebras
\beq
\Hom_{R^e}(P, \End_\k(M))\overset{\sim}\to \End_{R}(P\ox_RM).
\eeq
\label{ack!}
\end{proposition}

Strictly speaking, only the first portion of this proposition is given
in Weibel's text.  The compatibility of the cup product with the
Yoneda product, in the case that $M=N$, should certainly also be well
known.  One simply verifies that the map sending $f\in
\Hom_B(P,\End_\k(M))$ to the function
\beq
\ba{l}
r\ox r_1\ox\dots\ox r_n\ox m\\
\hspace{5mm}\mapsto (-1)^{|f|(n-|f|)}r\ox r_1\ox\dots\ox r_{n-|f|}\ox f(1\ox r_{|n|-|f|+1}\ox\dots\ox r_n\ox 1)(m)
\ea
\eeq
in $\End_R(P\ox_RM)$ is a morphism dg algebras.  The fact that the
proposed map is a quasi-isomorphism follows by commutativity of the diagram
\beq
\xymatrixcolsep{4mm}
\xym{
 & \End_R(P\ox_R M)\ar[dr]^\sim & \\
\Hom_{R^e}(P, \End_\k(M))\ar[ur]\ar[rr]^\cong& & \Hom_{R}(P\ox_RM,M).
}
\eeq

\begin{corollary}
There is a canonical isomorphism of graded vector spaces
$\mathrm{HH}(R,\Hom_\k(M,N))\cong \Ext_R(M,N)$ and isomorphism of
graded algebras $\mathrm{HH}(R,\End_\k(M))\cong \Ext_R(M,M)$.
\end{corollary}

Note that, for any projective bimodule resolution $P\to R$ with
diagonal quasi-isomorphism $\omega:P\to P\ox_R P$, the dg algebra
structure on $\Hom_{R^e}(P,\mathrm{End}(M))$ induces a dg algebra
structure on the complex $\Hom_{R}(P\ox_R M, M)$ by way of the
adjunction isomorphism of Proposition \ref{ack!}.  For functions
$f,g\in \Hom_{R}(P\ox_R M, M)$, the product $fg\in \Hom_{R}(P\ox_R M,
M)$ will be given by
\beq
fg:x\ox_R m\mapsto (-1)^{|g||\omega_1(x)|}f(\omega_1(x)\ox_A g(\omega_2(x)\ox_A m)),
\eeq
where the notation $\omega(x)=\omega_1(x)\ox_A \omega_2(x)$ is as in
Section \ref{DIA}.

\begin{corollary}
Let $M$ be any $R$-module, $P$ be any projective $R$-bimodule
resolution over $R$, and $\omega:P\to P\ox_R P$ be any $R^e$-linear
quasi-isomorphism.  Give $\Hom_{R^e}(P\ox_R M,M)$ the dg algebra
structure outlined above.  Then
\beq
\Ext_{R}(M,M)\cong \mathrm{H}(\Hom_{R}(P\ox_RM,M))
\eeq
as an algebra.
\label{cor15}
\end{corollary}
\begin{proof}
The algebra structure on $\Hom_{R^e}(P,\mathrm{End}(M))$ is defined so
that the isomorphism $\Hom_{R^e}(P,\mathrm{End}(M))\overset{\cong}\to
\Hom_R(P\ox_R M,M)$ is one of dg algebras.  In the case that $P$ is
the bar resolution, there is an isomorphism of algebras
\beq
\mathrm{H}(\Hom_R(P\ox_R M,M))\cong \mathrm{H}(\Hom_{R^e}(P,\mathrm{End}(M)))\cong \Ext_{R}(M,M)
\eeq
by Proposition \ref{ack!}.  The result now follows from the fact that
the cup product on
$\mathrm{H}(\Hom_{R^e}(P,\mathrm{End}(M)))=\mathrm{HH}(R,\End_\k(M))$
can be computed using any resolution and any quasi-isomorphism $\omega:P\to P\ox_A P$.
\end{proof}

Let us return to our analysis of the cohomology of smash products.  We
fix a Hopf algebra $\Gamma$ and $\Gamma$-module algebra $A$.  Before
giving the main results let us clarify a possible point of confusion.
\par
For any $A\#\Gamma$-modules $M$ and $N$ there is a standard way to
endow $\RHom_A(M,N)$ with a right $\Gamma$-module structure.  One
simply takes a $A\#\Gamma$-complex $Q$ which is a projective
resolution of $M$ over $A$ and defines the action on
$\RHom_A(M,N)=\Hom_A(Q,N)$ by
\beq
f\cdot \gamma(q):= S(\gamma_1)f(\gamma_2q).
\eeq
We would like to know that this $\Gamma$-module structures agree with
the $\Gamma$-module structure on $\RHom_{A^e}(A,\End_\k(M,N))$ given
in Section \ref{HCvDI}.

\begin{lemma}
Let $M$ and $N$ be modules over $A\#\Gamma$ and $K$ be a projective
bimodule resolution of $A$ satisfying conditions (I) and (II).  Then
the complex $K\ox_A M$ is a $A\#\Gamma$-complex under the diagonal
$\Gamma$-action and the quasi-isomorphism $K\ox_A M\to M$ is
$\Gamma$-linear.  Furthermore, the isomorphism
\beq
\Hom_{A^e}(K, \Hom_\k(M,N))\overset{\cong}\to \Hom_{A}(K\ox_AM,N).
\eeq
is one of complexes of right $\Gamma$-modules.
\label{somelems}
\end{lemma}

Since this lemma is not essential to the remainder of the paper, the
proof is deferred  to the appendix.  We now give the main results of
the section.

\begin{theorem}
Let $M$ and $N$ be $A\#\Gamma$-modules, $L$ be a projective resolution
of the trivial right $\Gamma$-module $\k$, and $K$ be a projective
$A$-bimodule resolution of $A$ satisfying (I) and (II).  Then there is
an isomorphism of chain complexes
\beq
\Hom_{A\#\Gamma}(K\# L^\uparrow\ox_{A\#\Gamma}M,N)\overset{\cong}\to \Hom_\Gamma(L,\Hom_A(K\ox_A M,N)).
\eeq
When the hypotheses of Theorem \ref{bigthm2} are satisfied, the isomorphism
\beq
\Hom_{A\#\Gamma}(K\# L^\uparrow\ox_{A\#\Gamma}M,M)\overset{\cong}\to \Hom_\Gamma(L,\Hom_A(K\ox_A M,M))
\eeq
is one of (non-necessarily-associative) dg algebras.
\label{bigthm3}
\end{theorem}

In the above statement, $\Hom_{A\#\Gamma}(K\#
L^\uparrow\ox_{A\#\Gamma}M,M)$ and $\Hom_A(K\ox_A M,M)$ are supposed
to have the algebra structures of Corollary \ref{cor15}.

\begin{proof}
From Proposition \ref{ack!}, Theorem \ref{bigthm1}, and the previous
lemma, we get a sequence of isomorphisms of chain complexes
\beq
\ba{rl}
\Hom_{A\#\Gamma}(K\# L^\uparrow\ox_{A\#\Gamma}M,N)&\cong \Hom_{(A\#\Gamma)^e}(K\# L^\uparrow,\Hom_\k(M,N))\\
&\cong \Hom_{\Gamma}(L,\Hom_{A^e}(K,\Hom_\k(M,N))\\
&\cong \Hom_\Gamma(L,\Hom_A(K\ox_A M, N)).
\ea
\eeq
Suppose now that $N=M$ and that the hypotheses of Theorem
\ref{bigthm2} are met.  Then the second isomorphism
\beq
\Hom_{(A\#\Gamma)^e}(K\# L^\uparrow,\Hom_\k(M,N))\cong \Hom_{\Gamma}(L,\Hom_{A^e}(K,\Hom_\k(M,N))
\eeq
is one of dg algebras by Theorem \ref{bigthm2}.  The isomorphisms
\beq
\Hom_{A\#\Gamma}(K\# L^\uparrow\ox_{A\#\Gamma}M,N)\cong \Hom_{(A\#\Gamma)^e}(K\# L^\uparrow,\Hom_\k(M,N))
\eeq
and
\beq
\Hom_{A^e}(K,\Hom_\k(M,N))\cong \Hom_A(K\ox_A M, N)
\eeq
are isomorphisms of dg algebras by the definition of the
multiplicative structures considered in Corollary \ref{cor15}.  It
follows that the final isomorphism
\beq
\Hom_{\Gamma}(L,\Hom_{A^e}(K,\Hom_\k(M,N))\cong \Hom_\Gamma(L,\Hom_A(K\ox_A M, N))
\eeq
is one of dg algebras as well.  Taking this all together gives the
proposed result.
\end{proof}

In the complex $\Hom_\Gamma(L,\Hom_A(K\ox_A M,N))$, we may replace
$K\ox_A M$ with any $\Gamma$-linear $A$-projective resolution $Q\to
M$.  The proof of this fact is the similar to the one given for
Corollary \ref{bigcor1}.  In the entire statement of Theorem
\ref{bigthm3} we could have also replaces $K$ with any
$A^e$-projective $\Gamma$-equivariant resolution $P\to A$ with an
equivariant diagonal, by Corollary \ref{bigcor2}.

\begin{corollary}
Let $M$ and $N$ be modules over $A\#\Gamma$, and $K$ and $L$ be as in
Theorem \ref{bigthm2}.  Then there is an isomorphism of graded vector spaces
\beq
\Ext_{A\#\Gamma}(M,N)\cong\mathrm{H}\big(\Hom_\Gamma(L,\Hom_A(K\ox_AM,N))\big)
\eeq
and an isomorphism of graded algebras
\beq
\Ext_{A\#\Gamma}(M,M)\cong\mathrm{H}\big(\Hom_\Gamma(L,\Hom_A(K\ox_AM,M))\big).
\eeq
\label{thmthmthms}
\end{corollary}

\begin{proof}
This follows by the isomorphism of dg algebras
\beq
\Hom_{A\#\Gamma}(K\# L^\uparrow\ox_{A\#\Gamma}M,M)\overset{\cong}\to \Hom_\Gamma(L,\Hom_A(K\ox_A M,M))
\eeq
of Theorem \ref{bigthm3} and the fact that
\beq
\Ext_{A\#\Gamma}(M,M)=\mathrm{H}\big(\Hom_{A\#\Gamma}(K\# L^\uparrow\ox_{A\#\Gamma}M,M)\big)
\eeq
as an algebra, by Corollary \ref{cor15}.
\end{proof}

Of course, from the Theorem we can also derive some multaplicative
spectral sequences converging to the $\mathrm{Ext}$ algebra
$\mathrm{Ext}_{A\#\Gamma}(M,M)$.  We define the filtrations $F^\Gamma$
and $F^A$ on each $\mathrm{Ext}_A(M,M)$ in the same manner as was done
at Notation \ref{filtnotes}.

\begin{corollary}
For any $A\#\Gamma$-module $M$, there are two multiplicative spectral sequences
\beq
E_2=\Ext_{\Gamma\text{-}\mathrm{mod}}(\k,\Ext_{A\text{-}\mathrm{mod}}(M,M))\Rightarrow \Ext_{A\#\Gamma\text{-}\mathrm{mod}}(M,M)
\eeq
and
\beq
'E_1=\Ext_{\Gamma\text{-}\mathrm{mod}}(\k,\RHom_{A\text{-}\mathrm{mod}}(M,M))\Rightarrow \Ext_{A\#\Gamma\text{-}\mathrm{mod}}(M,M)
\eeq
which converge to $\Ext_{A\#\Gamma\text{-}\mathrm{mod}}(M,M)$ as an
algebra.  In the case that the global dimension of $\Gamma$ is $\leq
1$ we have
\beq
\gr_\Gamma \Ext_{A\#\Gamma}(M,M)=\Ext_\Gamma(\k,\Ext_A(M,M))
\eeq
as an algebra.
\label{lastocor}
\end{corollary}
\begin{proof}
The two spectral sequences arise from considering the row and column
filtrations on the double complex $\Hom_\Gamma(L,\Hom_A(K\ox_AM,M))$.
The details are the same as those given in the proofs of Corollaries
\ref{spectrill} and \ref{spectrilll}.
\end{proof}

This corollary can be seen as a generalization of the
Lyndon-Hochschild-Serre spectral sequence in the following sense: if
$N$ and $G$ are groups, and $G$ acts on $N$ by automorphisms, we get
an action of $\Gamma=\k G$ on $A=\k N$.  We then have $A\# \Gamma=\k
(N\rtimes G)$, where $N\rtimes G$ denotes the semi-direct product.
When $M=\k$ is the trivial $N\rtimes G$ module, the multiplicative
spectral sequence
\beq
\Ext_\Gamma(\k,\Ext_A(M,M))=\mathrm{H}(G,\mathrm{H}(N,\k))\Rightarrow \Ext_{A\#\Gamma}(\k,\k)=\mathrm{H}(N\rtimes G,\k)
\eeq
of Corollary \ref{lastocor} is simply the Lyndon-Hochschild-Serre
spectral sequence.

\section*{Appendix}

\begin{proof}[Proof of Corollary \ref{bigcor2}]
Let $\mathscr{P}\to P$ be a projective resolution of $P$ as a complex
in $\mathrm{EQ}_\Gamma A^e$-mod.  Since the restriction functor
$\mathrm{EQ}_\Gamma A^e$-mod$\to A^e$-mod preserves projectives,
$\mathscr{P}$ will be a projective bimodule resolution of $A$ as well.
The tensor product of the quasi-isomorphism $\mathscr{P}\to P$ then
produces a quasi-isomorphism $\mathscr{P}\ox_A \mathscr{P}\to P\ox_A
P$ and we get a diagram
\beq
\xym{
\mathscr{P}\ox_A \mathscr{P}\ar[r] & P\ox_A P\\
\mathscr{P}\ar[u]\ar[r]\ar@{}[ur]|{\txt{$h\circlearrowleft$}} & P,\ar[u]
}
\eeq
which commutes up to a $\Gamma$-equivariant homotopy $h:\mathscr{P}\to
P\ox_A P[1]$.  The existence of this homotopy follow by projectivity
of $\mathscr{P}$ and the fact that the diagram commutes on homology.
We will then have a diagram
\beq
\xym{
\Hom_{A^e}(P,B)\ox \Hom_{A^e}(P,B)\ar[d]\ar[r]\ar[d]\ar@/_4.5pc/@{..>}[dd]_{mult} &\Hom_{A^e}(\mathscr{P},B)\ox\Hom_{A^e}(\mathscr{P},B)\ar[d]\ar@/^4.5pc/@{..>}[dd]^{mult}\\
\Hom_{A^e}(P\ox_A P, B)\ar[r]\ar[d] & \Hom_{A^e}(\mathscr{P}\ox_A \mathscr{P}, B)\ar[d]\\
\Hom_{A^e}(P,B)\ar[r]\ar@{}[ur]|{\txt{$h^\ast\circlearrowleft$}} & \Hom_{A^e}(\mathscr{P},B),
}\eeq
where the top square commutes and the bottom square commutes up to the
$\Gamma$-linear homotopy 
\beq
h^\ast:\Hom_{A^e}(P\ox_A P,B)\to\Hom_{A^e}(\mathscr{P},B)[1].
\eeq
To be clear, the top most vertical maps take a product of functions
$f\ox g$ to the function sending a monomial $x\ox_A y$ in $P\ox_A P$,
or $\mathscr{P}\ox_A \mathscr{P}$, to $(-1)^{|x||g|}f(x)g(y)$.  It
follows that the map $\Hom_{A^e}(P,B)\to \Hom_{A^e}(\mathscr{P},B)$ is
an algebra map, up to a homotopy, as is the induced map
\beq
\Hom_\Gamma(L,\Hom_{A^e}(P,B))\to \Hom_\Gamma(L,\Hom_{A^e}(\mathscr{P},B)).
\eeq
Since $P$ and $\mathscr{P}$ are projective $A$-bimodule resolutions of
$A$, both of these maps are also seen to be quasi-isomorphisms.
\par
Repeat the process with $K$ to get a resolution $\mathscr{K}\to K$ and
quasi-isomorphism
\beq
\Hom_\Gamma(L,\Hom_{A^e}(K,B))\to \Hom_\Gamma(L,\Hom_{A^e}(\mathscr{K},B))
\eeq
which is an algebra map up to a homotopy.  Finally, since both
$\mathscr{K}$ and $\mathscr{P}$ are projective resolutions of $A$ in
$\mathrm{EQ}_\Gamma A^e$, there is a $\Gamma$-equivariant
quasi-isomorphism $\mathscr{K}\to\mathscr{P}$.  We repeat the above
argument a third time to deduce a quasi-isomorphism
\beq
\Hom_\Gamma(L,\Hom_{A^e}(\mathscr{P},B))\to \Hom_\Gamma(L,\Hom_{A^e}(\mathscr{K},B))
\eeq
which is an algebra map up to a homotopy.  Taking $\mathscr{A}=
\Hom_\Gamma(L,\Hom_{A^e}(\mathscr{K},B))$ then provides the desired result.
\end{proof}

\begin{proof}[Proof of Lemma \ref{somelems}]
The first claim follows from the fact that $K$ is a left
$A\#\Gamma$-complex itself.  Now, for any $f$ in $\Hom_A(K\ox_AM,N)$
let $f_{\Hom}$ denote its image in $\Hom_{A^e}(K,\Hom_\k(M,N))$ under
the adjunction isomorphism of Proposition \ref{ack!}.  Then we have,
for any $\gamma\in \Gamma$, $x\in K$, and $m\in M$,
\beq
\ba{rl}
\big((f\gamma)_{\Hom}(x)\big)(m)&=f\gamma(x\ox_Am)\\
&=S(\gamma_1)f(\gamma_2x\ox_A\gamma_3m)\\
&=S(\gamma_1)f_{\Hom}(\gamma_2x)(\gamma_3m)\\
&=\big(S(\gamma_1)f_{\Hom}(\gamma_2x)\gamma_3)(m)\\
&=(f_{\Hom}\gamma(x))(m).
\ea
\eeq
So $f\gamma$ maps to $f_{\Hom}\gamma$ under the adjunction isomorphism
of the proof of Theorem \ref{ack!} and, consequently, the isomorphism
\beq
\RHom_{A^e}(A, \Hom_\k(M,N))\overset{\cong}\to \RHom_{A}(M,N)
\eeq
is $\Gamma$-linear.
\end{proof}

\def\cprime{$'$}

\end{document}